\documentclass{amsart}
\usepackage{amsmath}
\usepackage{amssymb}
\usepackage{amsthm}
\usepackage{epsfig}
\usepackage{amsfonts}
\usepackage{amscd}
\usepackage{mathrsfs}
\usepackage{enumerate}
\usepackage{latexsym}
\usepackage{url}

\newtheorem{theorem}{Theorem}[section]
\newtheorem{lemma}[theorem]{Lemma}
\newtheorem{proposition}[theorem]{Proposition}
\newtheorem{question}[theorem]{Question}

\theoremstyle{definition}
\newtheorem{definition}[theorem]{Definition}

\newtheorem{example}[theorem]{Example}

\theoremstyle{remark}
\newtheorem{remark}[theorem]{Remark}
\numberwithin{equation}{section}

\begin{document}

\title[Connectedness modulo an ideal]{Connectedness modulo an ideal}

\author{M.R. Koushesh}

\address{Department of Mathematical Sciences, Isfahan University of Technology, Isfahan 84156--83111, Iran}

\address{School of Mathematics, Institute for Research in Fundamental Sciences (IPM), P.O. Box: 19395--5746, Tehran, Iran}

\email{koushesh@cc.iut.ac.ir}

\thanks{This research was in part supported by a grant from IPM (No. 93030418).}



\begin{abstract}
For a topological space $X$ and an ideal $\mathscr{H}$ of subsets of $X$ we introduce the notion of connectedness modulo $\mathscr{H}$. This notion of connectedness naturally generalizes the notion of connectedness in its usual sense. In the case when $X$ is completely regular, we introduce a subspace $\gamma_{\mathscr H} X$ of the Stone--\v{C}ech compactification $\beta X$ of $X$, such that connectedness modulo ${\mathscr H}$ is equivalent to connectedness of $\beta X\setminus\gamma_{\mathscr H} X$. In particular, we prove that when ${\mathscr H}$ is the ideal generated by the collection of all open subspaces of $X$ with pseudocompact closure, then $X$ is connected modulo ${\mathscr H}$ if and only if $\mathrm{cl}_{\beta X}(\beta X\setminus\upsilon X)$ is connected, and when $X$ is normal and ${\mathscr H}$ is the ideal generated by the collection of all closed realcompact subspaces of $X$, then $X$ is connected modulo ${\mathscr H}$ if and only if $\mathrm{cl}_{\beta X}(\upsilon X\setminus X)$ is connected. Here $\upsilon X$ is the Hewitt realcompactification of $X$.
\end{abstract}

\maketitle

\section{Introduction}

Throughout this article by \textit{completely regular} we mean completely regular and Hausdorff (that is, Tychonoff).

An \textit{ideal} $\mathscr{H}$ in a set $X$ is a non-empty collection of subsets of $X$ such that
\begin{itemize}
  \item if a set $A$ is contained in an element of $\mathscr{H}$ then $A$ is in $\mathscr{H}$,
  \item if $G$ and $H$ are in $\mathscr{H}$ then so is their union $G\cup H$.
\end{itemize}
Intuitively, an ideal is a collection of subsets that are considered to be ``small''. Let ${\mathscr A}$ be a collection of subsets of a set $X$. The \textit{ideal in $X$ generated by ${\mathscr A}$}, denoted by $\langle{\mathscr A}\rangle$, is the intersection of all ideals in $X$ which contain ${\mathscr A}$. It is easy to check that $\langle{\mathscr A}\rangle$ is the collection of all subsets of finite unions of elements from ${\mathscr A}$.

The purpose of this article is to present a natural generalization of the notion of connectedness in topological spaces. Recall that a space $X$ is said to be \textit{connected} if there is no continuous $2$-valued mapping on $X$, that is, there is no continuous mapping $f:X\rightarrow[0,1]$ such that \begin{itemize}
  \item $f^{-1}(0)$ and $f^{-1}(1)$ are neither empty,
  \item $X\setminus(f^{-1}(0)\cup f^{-1}(1))$ is empty.
\end{itemize}
Now, to generalize, we replace ``emptyness'' in the above definition by ``smallness'', that is, ``being an element of an ideal in $X$''. More precisely, for a space $X$ and an ideal $\mathscr{H}$ in $X$, we say that \textit{$X$ is connected modulo ${\mathscr H}$} if there is no continuous mapping $f:X\rightarrow[0,1]$ such that
\begin{itemize}
  \item $f^{-1}(0)$ and $f^{-1}(1)$ are neither in $\mathscr{H}$,
  \item $X\setminus(f^{-1}(0)\cup f^{-1}(1))$ is in $\mathscr{H}$.
\end{itemize}
From this simple definition much follows. Indeed, most standard facts about connectedness have counterparts in this context. Note that for the trivial ideal $\{\emptyset\}$, connectedness modulo an ideal coincides with connectedness in its usual sense.

This article is organized as follows. There are two main sections. In Section \ref{OHJ} we study connectedness modulo an ideal ${\mathscr H}$ with no particular restriction on ${\mathscr H}$; in Section \ref{HGGGF} we deal with particular examples of ${\mathscr H}$. In Section \ref{OHJ} we generalize the standard known facts about connectedness to our new setting. This include theorems on preservation of connectedness under continuous mappings, formation of unions and taking closures. Further, for a completely regular space $X$ and an ideal ${\mathscr H}$ of subsets of $X$, we introduce a subspace $\lambda_{\mathscr H} X$ of the Stone--\v{C}ech compactification $\beta X$ of $X$, such that connectedness modulo the ideal ${\mathscr H}$ of $X$ is equivalent to connectedness of $\beta X\setminus\lambda_{\mathscr H} X$. In Section \ref{HGGGF} we consider particular examples of the ideal ${\mathscr H}$. In particular, we show that if $X$ is completely regular and ${\mathscr H}$ is the ideal generated by the collection of all open subspaces of $X$ whose closures are pseudocompact, then $X$ is connected modulo ${\mathscr H}$ if and only if $\mathrm{cl}_{\beta X}(\beta X\setminus\upsilon X)$ is connected, and if $X$ is normal and ${\mathscr H}$ is the ideal generated by the collection of all closed realcompact subspaces of $X$, then $X$ is connected modulo ${\mathscr H}$ if and only if $\mathrm{cl}_{\beta X}(\upsilon X\setminus X)$ is connected. Here $\upsilon X$ denotes the Hewitt realcompactification of $X$.

We now review some notation and certain known facts. For undefined terms and notation we refer the reader to the standard texts \cite{E}, \cite{GJ} and \cite{PW}.

Let $X$ be a space. A \textit{zero-set in $X$} (\textit{cozero-set in $X$}, respectively) is a set of the form $f^{-1}(0)$ ($X\setminus f^{-1}(0)$, respectively) where $f:X\rightarrow [0,1]$ is a continuous mapping. For a continuous mapping $f:X\rightarrow [0,1]$, the \textit{zero-set of $f$} (\textit{cozero-set of $f$}, respectively) is $f^{-1}(0)$ ($X\setminus f^{-1}(0)$, respectively). The set of all zero-sets of $X$ (cozero-sets of $X$, respectively) is denoted by $\mathrm{Z}(X)$ ($\mathrm{Coz}(X)$, respectively).

\subsubsection*{\textbf{The Stone--\v{C}ech compactification.}} Let $X$ be a completely regular space. By a \textit{compactification of $X$} we mean a compact Hausdorff space which contains $X$ as a dense subspace. The \textit{Stone--\v{C}ech compactification of $X$}, denoted by $\beta X$, is the compactification of $X$ which is characterized among all compactifications of $X$ by the fact that every continuous mapping $f:X\rightarrow K$, where  $K$ is a compact Hausdorff space (or $[0,1]$), is extendable to a (unique) continuous mapping over $\beta X$; we denote this continuous extension of $f$ by $f_\beta$. The Stone--\v{C}ech compactification of a completely regular space always exists.

\subsubsection*{\textbf{The Hewitt realcompactification.}} A space is called \textit{realcompact} if it is homeomorphic to a closed subspace of some product $\mathbb{R}^\alpha$. Let $X$ be a completely regular space. A \textit{realcompactification of $X$} is a realcompact space which contains $X$ as a dense subspace. The \textit{Hewitt realcompactification of $X$}, denoted by $\upsilon X$, is the realcompactification of $X$ which is characterized among all realcompactifications of $X$ by the fact that every continuous mapping $f:X\rightarrow R$ from $X$ to a realcompact space $R$ (or $\mathbb{R}$) is continuously extendable over $\upsilon X$. The Hewitt realcompactification of a completely regular space always exists. We may assume that $\upsilon X$ is a subspace of $\beta X$. Observe that $\upsilon X=X$ if and only if $X$ is realcompact.

\section{General results}\label{OHJ}

In this section we study connectedness modulo an ideal with no particular restriction on the ideal. Examples in which the ideal is specified are considered in the follow-up section.

\subsection{$\mathscr{H}$-connectedness; the definition}\label{OJ}

In this subsection we provide basic definitions which we will refer to throughout this article.

\begin{definition}\label{RUA}
Let $X$ be a space and let $\mathscr{H}$ be an ideal in $X$. A mapping $f:X\rightarrow[0,1]$ is called \textit{$2$-valued modulo $\mathscr{H}$} if
\begin{itemize}
  \item $f^{-1}(0)$ and $f^{-1}(1)$ are neither in $\mathscr{H}$.
  \item $X\setminus(f^{-1}(0)\cup f^{-1}(1))$ is in $\mathscr{H}$.
\end{itemize}
The space $X$ is said to be \textit{connected modulo ${\mathscr H}$} (or \textit{${\mathscr H}$-connected}) if there is no continuous mapping $f:X\rightarrow[0,1]$ which is $2$-valued modulo $\mathscr{H}$.
\end{definition}

The following theorem is to provide a necessary and sufficient condition for a space to be connected modulo an ideal.

Recall that two subsets $C$ and $D$ of a space $X$ are said to be \textit{completely separated in $X$} if there is a continuous mapping $f:X\rightarrow[0,1]$ which is $0$ on $C$ and $1$ on $D$.

\begin{definition}\label{RUA}
Let $X$ be a space and let $\mathscr{H}$ be an ideal in $X$. A \textit{separation for $X$ modulo $\mathscr{H}$} is a pair $C$ and $D$ of completely separated subsets of $X$ such that
\begin{itemize}
  \item $C$ and $D$ are neither in $\mathscr{H}$.
  \item $X\setminus(C\cup D)$ is in $\mathscr{H}$.
\end{itemize}
\end{definition}

\begin{theorem}\label{KJHG}
Let $X$ be a space and let $\mathscr{H}$ be an ideal in $X$. The following are equivalent.
\begin{itemize}
\item[\rm(1)] $X$ is connected modulo $\mathscr{H}$.
\item[\rm(2)] There is no separation for $X$ modulo $\mathscr{H}$.
\end{itemize}
\end{theorem}

\begin{proof}
(1) \textit{implies} (2). Suppose that there is a separation $C$ and $D$ for $X$ modulo $\mathscr{H}$. Let $f:X\rightarrow[0,1]$ be continuous with
\[f|_C=\mathbf{0}\quad\text{and}\quad f|_D=\mathbf{1},\]
where $\mathbf{r}$ denotes the mapping assigning the real number $r$ to every element in its domain. Then $f^{-1}(0)$ is not in $\mathscr{H}$, as it contains $C$ and $C$ is not in $\mathscr{H}$. Similarly $f^{-1}(1)$ is not in $\mathscr{H}$. Also, $X\setminus(f^{-1}(0)\cup f^{-1}(1))$ is in $\mathscr{H}$, as it is contained in $X\setminus(C\cup D)$ and the latter is in $\mathscr{H}$. That is, $f$ is $2$-valued modulo $\mathscr{H}$. Thus $X$ is not connected modulo $\mathscr{H}$.

(2) \textit{implies} (1). Suppose that $X$ is not connected modulo $\mathscr{H}$. Let $f:X\rightarrow[0,1]$ be a continuous mapping which is $2$-valued modulo $\mathscr{H}$. Let
\[C=f^{-1}(0)\quad\text{and}\quad D=f^{-1}(1).\]
Then $C$ and $D$ is a pair of completely separated (by $f$) subsets of $X$ and clearly constitutes a separation for $X$ modulo $\mathscr{H}$.
\end{proof}

The following example shows that the concept ``connectedness modulo an ideal'' indeed generalizes ``connectedness'' in the usual sense.

\begin{example}\label{HGF}
Let $X$ be a space and let $\mathscr{H}$ be the ideal in $X$ consisting of the empty set $\emptyset$ alone. It is clear that in this case the notions ``a $2$-valued mapping modulo $\mathscr{H}$'' and ``a separation modulo $\mathscr{H}$'' coincide, respectively, with the notions ``a $2$-valued mapping'' and ``a separation'' in their usual senses. In particular, in this case the space $X$ is connected modulo $\mathscr{H}$ if and only if it is connected.
\end{example}

There may indeed exist ideals, other than the trivial ideal $\{\emptyset\}$ itself, for which the two notions ``connectedness modulo an ideal'' and ``connectedness'' coincide; this will be shown in the following two example.

\begin{example}\label{OFD}
Let $X$ be a $T_1$-space with no isolated points and let $\mathfrak{Fin}(X)$ be the ideal in $X$ consisting of all finite subsets of $X$. We verify that a continuous mapping $f:X\rightarrow[0,1]$ is $2$-valued modulo $\mathfrak{Fin}(X)$ if and only if it is $2$-valued in the usual sense. This, in particular, will show that connectedness modulo $\mathfrak{Fin}(X)$ coincides with connectedness. Let $f:X\rightarrow[0,1]$ be a continuous mapping which is $2$-valued modulo $\mathfrak{Fin}(X)$. That is, $f^{-1}(0)$ and $f^{-1}(1)$ are infinite while the remainder
\[U=X\setminus\big(f^{-1}(0)\cup f^{-1}(1)\big)=f^{-1}\big((0,1)\big)\]
is finite. But $U$ is open in $X$, and therefore it is infinite if it is non-empty, as $X$ is $T_1$ and it has no isolated points. Thus $U$ is empty. That is, $f$ is $2$-valued. Next, let $g:X\rightarrow[0,1]$ be a continuous $2$-valued mapping. Then $g^{-1}(0)$ is open in $X$ and it is non-empty, therefore, it is infinite. Similarly, $g^{-1}(1)$ is infinite. Thus $g$ is $2$-valued modulo $\mathfrak{Fin}(X)$.
\end{example}

\begin{example}\label{JKHJS}
Let $X$ be a space and let $\mathscr{H}$ be the collection of all nowhere dense subsets of $X$, that is, the collection of all subsets $H$ of $X$ such that $\mathrm{int}_X\mathrm{cl}_XH$ is empty. It is easy to check that $\mathscr{H}$ is an ideal in $X$. We verify that a continuous mapping $f:X\rightarrow[0,1]$ is $2$-valued modulo $\mathscr{H}$ if and only if it is $2$-valued in the usual sense. Let $f:X\rightarrow[0,1]$ be a continuous mapping which is $2$-valued modulo $\mathscr{H}$. Let
\[A=X\setminus\big(f^{-1}(0)\cup f^{-1}(1)\big).\]
Then $A$ is in $\mathscr{H}$. Note that $A$ is contained in $\mathrm{cl}_XA$ and thus in $\mathrm{int}_X\mathrm{cl}_XA$, since $A$ is open in $X$. But then $A$ is empty. Since neither of $f^{-1}(0)$ and $f^{-1}(1)$ is in $\mathscr{H}$, neither is empty in particular. That is $f$ is a $2$-valued mapping. That every $2$-valued continuous mapping is $2$-valued modulo $\mathscr{H}$ is clear.
\end{example}

As the above examples suggest, a space which is connected modulo an ideal may still remain connected as the ideal enlarges. The natural question which arises at this point is ``how big this enlargement can be?'' The following is to provide an answer to this question.

\begin{theorem}\label{HGFF}
Let $X$ be a space and let $\mathscr{H}$ be an ideal in $X$. Suppose that $X$ is connected modulo $\mathscr{H}$. Then, there is a maximal (with respect to the set-theoretic inclusion $\subseteq$) ideal $\mathscr{M}$ in $X$ which contains $\mathscr{H}$ and $X$ is connected modulo $\mathscr{M}$.
\end{theorem}

\begin{proof}
Consider the family
\[\mathbb{H}=\{\mathscr{G}:\mathscr{G}\mbox{ is an ideal in $X$ containing $\mathscr{H}$ and $X$ is connected modulo $\mathscr{H}$}\},\]
partially ordered with the set-theoretic inclusion $\subseteq$. We show that $\mathbb{H}$ has a maximal element. To show this, by Zorn's lemma, it suffices to check that every non-empty linearly ordered (by $\subseteq$) subfamily of $\mathbb{H}$ has an upper bounded in $\mathbb{H}$. Let $\mathbb{G}$ be a non-empty linearly ordered subfamily of $\mathbb{H}$. Let
\[\mathscr{H}^*=\bigcup_{\mathscr{G}\in\mathbb{G}}\mathscr{G}.\]
It is easy to check that $\mathscr{H}^*$ is an ideal in $X$. (To check that $\mathscr{H}^*$ is closed under formation of finite unions, let $G$ and $H$ be two elements of $\mathscr{H}^*$. Then (by the definition of $\mathscr{H}^*$) there are elements $\mathscr{G}$ and $\mathscr{H}$ of $\mathbb{G}$ having $G$ and $H$, respectively. Since $\mathbb{G}$ is linearly ordered, either $\mathscr{G}$ contains $\mathscr{H}$, or $\mathscr{H}$ contains $\mathscr{G}$. But then $G$ and $H$ are both in either $\mathscr{G}$ or $\mathscr{H}$ and thus so is their union $G\cup H$. Therefore $G\cup H$ is in $\mathscr{H}^*$.) Also, $\mathscr{H}^*$ contains $\mathscr{H}$, as every element of $\mathscr{G}$ does (and $\mathscr{G}$ is non-empty). We now check that $X$ is connected modulo $\mathscr{H}^*$. Suppose otherwise. Then, there is a continuous mapping $f:X\rightarrow[0,1]$ which is $2$-valued modulo $\mathscr{H}^*$. The set
\[X\setminus\big(f^{-1}(0)\cup f^{-1}(1)\big)\]
is in $\mathscr{H}^*$ and thus is in $\mathscr{I}$ for some $\mathscr{I}$ in $\mathbb{G}$. But then $f$ is $2$-valued modulo $\mathscr{I}$, as $f^{-1}(0)$ and $f^{-1}(1)$ are neither in $\mathscr{H}^*$ and thus are neither in $\mathscr{I}$. That is, $X$ is not connected modulo $\mathscr{I}$. This contradiction shows that $X$ is connected modulo $\mathscr{H}^*$. Therefore $\mathscr{H}^*$ is an element of $\mathbb{H}$. It is clear that $\mathscr{H}^*$ contains every element of $\mathbb{G}$.
\end{proof}

Even in simple cases (such as the case when $X$ is the real line $\mathbb{R}$ and $\mathscr{H}$ is the trivial ideal $\{\emptyset\}$) we do not know any satisfactory description of an ideal whose existence is guaranteed by Theorem \ref{HGFF}. Let us formally state this below for possible future reference.

\begin{question}\label{JHGFD}
Let $X$ be a space and let $\mathscr{H}$ be an ideal in $X$ such that $X$ is connected modulo $\mathscr{H}$. Describe the elements of a maximal ideal $\mathscr{M}$ of $X$ which contains $\mathscr{H}$ and $X$ is connected modulo $\mathscr{M}$.
\end{question}

\subsection{Images of $\mathscr{H}$-connected spaces}

It is well known that any continuous image of a connected space is connected. This subsection is to provide a counterpart for this result in the new context.

Let $X$ and $Y$ be sets and let $\phi:X\rightarrow Y$ be a mapping. For an ideal $\mathscr{H}$ in $Y$ we denote
\[\phi^{-1}({\mathscr H})=\big\{A\subseteq X:\phi(A)\in{\mathscr H}\big\}.\]
It is easy to check that $\phi^{-1}({\mathscr H})$ is indeed an ideal in $X$.

\begin{theorem}\label{KJG}
Let $X$ and $Y$ be spaces and let $\phi:X\rightarrow Y$ be a continuous surjection. Let $\mathscr{H}$ be an ideal in $Y$. Then, if $X$ is connected modulo $\phi^{-1}({\mathscr H})$ then $Y$ is connected modulo $\mathscr{H}$.
\end{theorem}

\begin{proof}
Suppose that $Y$ is not connected modulo $\mathscr{H}$; we show that $X$ is then not connected modulo $\phi^{-1}({\mathscr H})$. Let $f:Y\rightarrow[0,1]$ be a continuous mapping which is $2$-valued modulo $\mathscr{H}$. The composition $f\phi:X\rightarrow[0,1]$ is continuous and, as we now check, is $2$-valued modulo $\phi^{-1}({\mathscr H})$. Observe that $(f\phi)^{-1}(0)$ is not in $\phi^{-1}({\mathscr H})$; as otherwise, \[\phi\big((f\phi)^{-1}(0)\big)=\phi\big(\phi^{-1}\big(f^{-1}(0)\big)\big)\]
is in ${\mathscr H}$. But this is not possible, as
\[\phi\big(\phi^{-1}\big(f^{-1}(0)\big)\big)=f^{-1}(0),\]
since $\phi$ is surjective. Similarly, $(f\phi)^{-1}(1)$ is not in $\phi^{-1}({\mathscr H})$. Also,
\[A=X\setminus\big((f\phi)^{-1}(0)\cup (f\phi)^{-1}(1)\big)\]
is in $\phi^{-1}({\mathscr H})$, as $Y\setminus(f^{-1}(0)\cup f^{-1}(1))$ is in $\mathscr{H}$,
\[\phi(A)=\phi\big(X\setminus\big(\phi^{-1}\big(f^{-1}(0)\big)\cup\phi^{-1}\big(f^{-1}(1)\big)\big)\big)=\phi\big(\phi^{-1}\big(Y\setminus\big(f^{-1}(0)\cup f^{-1}(1)\big)\big)\big)\]
and
\[\phi\big(\phi^{-1}\big(Y\setminus\big(f^{-1}(0)\cup f^{-1}(1)\big)\big)\big)=Y\setminus\big(f^{-1}(0)\cup f^{-1}(1)\big),\]
since $\phi$ is surjective. That is $f$ is $2$-valued modulo $\phi^{-1}({\mathscr H})$. The space $X$ is therefore not connected modulo $\phi^{-1}({\mathscr H})$.
\end{proof}

\subsection{Unions of $\mathscr{H}$-connected subspaces}

The union of a collection of connected subspaces of a space is connected, provided that their intersection is non-empty. Our purpose here is to provide a counterpart for this well known result in the new context.

Let $X$ be a set and let $\mathscr{H}$ be an ideal in $X$. For any subset $A$ of $X$ denote
\[{\mathscr H}|_A=\{H\cap A:H\in{\mathscr H}\}.\]
It is easy to check that ${\mathscr H}|_A$ is an ideal in $A$.

\begin{lemma}\label{JJHG}
Let $X$ be a space and let $\mathscr{H}$ be an ideal in $X$. Let $A$ be a subspace of $X$ which is connected modulo ${\mathscr H}|_A$. Let $f:X\rightarrow[0,1]$ be a continuous mapping which is $2$-valued modulo $\mathscr{H}$. Then either
\[A\cap f^{-1}(0)\quad\text{or}\quad A\cap f^{-1}(1)\]
is in $\mathscr{H}$.
\end{lemma}

\begin{proof}
Consider the continuous mapping $f|_A:A\rightarrow[0,1]$. Observe that the set
\[X\setminus\big(f^{-1}(0)\cup f^{-1}(1)\big)\]
is in $\mathscr{H}$ and thus
\[A\cap\big(X\setminus\big(f^{-1}(0)\cup f^{-1}(1)\big)\big)\]
is in $\mathscr{H}|_A$. Note that
\[A\setminus\big((f|_A)^{-1}(0)\cup (f|_A)^{-1}(1)\big)=A\cap\big(X\setminus\big(f^{-1}(0)\cup f^{-1}(1)\big)\big),\]
as
\[(f|_A)^{-1}(0)=A\cap f^{-1}(0)\quad\text{and}\quad(f|_A)^{-1}(1)=A\cap f^{-1}(1).\]
But $f|_A$ is not $2$-valued modulo $\mathscr{H}|_A$, as $A$ is connected modulo ${\mathscr H}|_A$. Therefore, either \[(f|_A)^{-1}(0)\quad\text{or}\quad(f|_A)^{-1}(1)\]
is in $\mathscr{H}|_A$. It is clear that $\mathscr{H}|_A$ is contained in $\mathscr{H}$.
\end{proof}

\begin{theorem}\label{HFDF}
Let $X$ be a space and let $\mathscr{H}$ be an ideal in $X$. Suppose that
\[X=X_1\cup\cdots\cup X_n,\]
where $X_i$ is connected modulo ${\mathscr H}|_{X_i}$ for each $i=1,\dots,n$. Suppose that
\[A=X_1\cap\cdots\cap X_n\]
is connected modulo ${\mathscr H}|_A$ and is not in ${\mathscr H}$. Then $X$ is connected modulo ${\mathscr H}$.
\end{theorem}

\begin{proof}
Suppose to the contrary that $X$ is not connected modulo ${\mathscr H}$. Suppose that $f:X\rightarrow[0,1]$ is a continuous mapping which is $2$-valued modulo $\mathscr{H}$. Since $A$ is connected modulo ${\mathscr H}|_A$, by Lemma \ref{JJHG}, either
\[A\cap f^{-1}(0)\quad\text{or}\quad A\cap f^{-1}(1),\]
say the latter, is in $\mathscr{H}$. Similarly, for each $i=1,\dots,n$, either
\[X_i\cap f^{-1}(0)\quad\text{or}\quad X_i\cap f^{-1}(1)\]
is in $\mathscr{H}$. Fix some $i=1,\dots,n$. Suppose that $X_i\cap f^{-1}(0)$ is in $\mathscr{H}$. Then $A\cap f^{-1}(0)$ is also in $\mathscr{H}$, as it is contained in $X_i\cap f^{-1}(0)$. Also, $A\setminus(f^{-1}(0)\cup f^{-1}(1))$ is in $\mathscr{H}$, as it is contained in $X\setminus(f^{-1}(0)\cup f^{-1}(1))$ and the latter is in $\mathscr{H}$. But then
\[A=\big[A\cap f^{-1}(0)\big]\cup\big[A\cap f^{-1}(1)\big]\cup\big[A\setminus\big(f^{-1}(0)\cup f^{-1}(1)\big)\big]\]
is in $\mathscr{H}$, which contradicts our assumption. Therefore $X_i\cap f^{-1}(1)$ is in $\mathscr{H}$ for each $i=1,\dots,n$, and thus so is their union $f^{-1}(1)$. This contradiction shows that $X$ is connected modulo ${\mathscr H}$.
\end{proof}

\begin{remark}\label{HFF}
Theorem \ref{HFDF} fails if the number of $X_i$'s is infinite; this will be illustrated in Example \ref{HGGG}.
\end{remark}

The following theorem generalizes Theorem \ref{HFDF}. We omit the proof, as it is analogous to the one we have already given for Theorem \ref{HFDF}.

Let $X$ be a set and let $\kappa$ be an infinite cardinal. An ideal $\mathscr{H}$ in $X$ is called \textit{$\kappa$-complete} if for every subcollection $\mathscr{G}$ of $\mathscr{H}$ of cardinality $<\kappa$ the union $\bigcup\mathscr{G}$ is in $\mathscr{H}$.

\begin{theorem}\label{LJF}
Let $X$ be a space and let $\mathscr{H}$ be a $\kappa$-complete ideal in $X$. Suppose that
\[X=\bigcup_{i<\kappa} X_i,\]
where $X_i$ is connected modulo ${\mathscr H}|_{X_i}$ for each $i<\kappa$. Suppose that
\[A=\bigcap_{i<\kappa} X_i\]
is connected modulo ${\mathscr H}|_A$ and is not in ${\mathscr H}$. Then $X$ is connected modulo ${\mathscr H}$.
\end{theorem}

\subsection{Closures of $\mathscr{H}$-connected subspaces}

It is well known that the closure of a connected subspace of a space is connected. In the following we provide a counterpart for this result in the new context.

Let $X$ be a space and let $\mathscr{H}$ be an ideal in $X$. Define
\[\mathrm{cl}_X{\mathscr H}=\langle\mathrm{cl}_XH:H\in{\mathscr H}\rangle.\]

\begin{theorem}\label{HGF}
Let $X$ be a space and let $\mathscr{H}$ be an ideal in $X$. For a subset $A$ of $X$ suppose that
\[\mathrm{cl}_XA=X\quad\text{and}\quad\mathrm{cl}_X({\mathscr H}|_A)={\mathscr H}.\]
Then, if $A$ is connected modulo $\mathscr{H}|_A$ then $X$ is connected modulo ${\mathscr H}$.
\end{theorem}

\begin{proof}
Suppose that $A$ is connected modulo $\mathscr{H}|_A$. Suppose to the contrary that $X$ is not connected modulo ${\mathscr H}$. Let $f:X\rightarrow[0,1]$ be a continuous mapping which is $2$-valued modulo $\mathscr{H}$. The mapping $f|_A:A\rightarrow[0,1]$ cannot be $2$-valued modulo $\mathscr{H}|_A$. Therefore, since
\[A\setminus\big((f|_A)^{-1}(0)\cup (f|_A)^{-1}(1)\big)=A\cap\big(X\setminus\big(f^{-1}(0)\cup f^{-1}(1)\big)\big)\]
is in $\mathscr{H}|_A$, as $X\setminus(f^{-1}(0)\cup f^{-1}(1))$ is in $\mathscr{H}$, either
\[A\cap f^{-1}(0)=(f|_A)^{-1}(0)\quad\text{or}\quad A\cap f^{-1}(1)=(f|_A)^{-1}(1),\]
say the latter, is in $\mathscr{H}|_A$. But then
\[A\cap f^{-1}\big((0,1]\big)=\big[A\cap f^{-1}(1)\big]\cup\big[A\setminus\big(f^{-1}(0)\cup f^{-1}(1)\big)\big]\]
is in $\mathscr{H}|_A$. Thus $\mathrm{cl}_X(A\cap f^{-1}((0,1]))$ is in $\mathrm{cl}_X(\mathscr{H}|_A)$. Observe that
\[\mathrm{cl}_X\big(A\cap f^{-1}\big((0,1]\big)\big)=\mathrm{cl}_Xf^{-1}\big((0,1]\big),\]
as $f^{-1}((0,1])$ is open in $X$ and $A$ is dense in $X$. Therefore $\mathrm{cl}_Xf^{-1}((0,1])$ is in $\mathrm{cl}_X(\mathscr{H}|_A)$, and thus so is its subset $f^{-1}(1)$. Since $\mathrm{cl}_X({\mathscr H}|_A)={\mathscr H}$, it follows that $f^{-1}(1)$ is in ${\mathscr H}$. This contradiction proves the theorem.
\end{proof}

\subsection{Products of $\mathscr{H}$-connected spaces}

As it is well known, the product of a collection of connected spaces remains connected. We do not know how this can be formulated in the new context, assuming indeed that such a formulation at all exists. (See Examples \ref{HGHFDF} and \ref{FFD} for related results.) For possible future reference we record this formally as a question.

\begin{question}
Let $X$ and $Y$ be spaces and let $\mathscr{H}$ and $\mathscr{G}$ be ideals in $X$ and $Y$, respectively. Define
\[\mathscr{H}\times\mathscr{G}=\langle H\times G:H\in\mathscr{H}\mbox{ and }G\in\mathscr{G}\rangle.\]
Suppose that $X$ and $Y$ are connected modulo $\mathscr{H}$ and $\mathscr{G}$, respectively. Is it then true that $X\times Y$ is connected modulo $\mathscr{H}\times\mathscr{G}$?
\end{question}

\begin{remark}
The definition we have given for the product of two ideals is not the standard one, though, in this context it is perhaps the most appropriate one.
\end{remark}

\subsection{$\mathscr{H}$-connectedness in completely regular spaces}

In this subsection we show that connectedness modulo an ideal ${\mathscr H}$ of subsets of a completely regular space $X$ is equivalent to connectedness of a certain subspace of the Stone--\v{C}ech compactification $\beta X$ of $X$ naturally associated with ${\mathscr H}$. This characterization is particularly useful when we deal with specific examples in the subsequent section.

For a completely regular space $X$ and an ideal ${\mathscr H}$ in $X$ we need to introduce the following subspace $\lambda_{\mathscr H} X$ of its Stone--\v{C}ech compactification $\beta X$.

\begin{definition}\label{RRA}
Let $X$ be a completely regular space and let ${\mathscr H}$ be an ideal in $X$. Define
\[\lambda_{\mathscr H} X=\bigcup\{\mathrm{int}_{\beta X}\mathrm{cl}_{\beta X}A:\mathrm{cl}_XA\in{\mathscr H}\}.\]
\end{definition}

Recall that any two disjoint zero-sets $S$ and $T$ is a space $X$ are completely separated; for if $S=\mathrm{Z}(f)$ and $T=\mathrm{Z}(g)$, where $f,g:X\rightarrow[0,1]$ are continuous mappings, then the mapping \[h=\frac{f}{f+g}:X\longrightarrow[0,1]\]
is continuous and is $0$ on $S$ and $1$ on $T$. Indeed, $S=h^{-1}(0)$ and $T=h^{-1}(1)$.

\begin{theorem}\label{BBV}
Let $X$ be a completely regular space and let $\mathscr{H}$ be an ideal in $X$. The following are equivalent.
\begin{itemize}
\item[\rm(1)] $X$ is connected modulo $\mathscr{H}$.
\item[\rm(2)] $\beta X\setminus\lambda_{\mathscr H} X$ is connected.
\end{itemize}
\end{theorem}

\begin{proof}
(1) \textit{implies} (2). Suppose that $\beta X\setminus\lambda_{\mathscr H} X$ is not connected. We show that $X$ is not connected modulo $\mathscr{H}$. Let $E$ and $G$ be a separation for $\beta X\setminus\lambda_{\mathscr H} X$. Note that $E$ and $G$ are closed in $\beta X$, as they are closed in $\beta X\setminus\lambda_{\mathscr H} X$ and the latter is closed in $\beta X$, since $\lambda_{\mathscr H} X$ is open in $\beta X$ by its definition. Since $\beta X$ is normal, by the Urysohn lemma, there exists a continuous mapping $f:\beta X\rightarrow[0,1]$ such that
\[f|_E=\textbf{0}\quad\text{and}\quad f|_G=\textbf{1}.\]
Let
\[C=X\cap f^{-1}\big([0,1/3]\big)\quad\text{and}\quad D=X\cap f^{-1}\big([2/3,1]\big).\]
We verify that the pair $C$ and $D$ constitutes a separation for $X$ modulo $\mathscr{H}$; Theorem \ref{KJHG} will then conclude the proof. Note that $C$ and $D$ are disjoint zero-sets of $X$ and therefore are completely separated in $X$. We check that $C$ is not in $\mathscr{H}$; the same proof will show that $D$ is not in $\mathscr{H}$ either. Suppose otherwise. Let
\[A=X\cap f^{-1}\big([0,1/3)\big).\]
Then
\[\mathrm{int}_{\beta X}\mathrm{cl}_{\beta X}A\subseteq\lambda_{\mathscr H} X\]
by the definition of $\lambda_{\mathscr H} X$, as $\mathrm{cl}_XA$ is contained in $C$. But then $E$ is contained in $\lambda_{\mathscr H} X$, as
\[E\subseteq f^{-1}\big([0,1/3)\big)\]
and
\[f^{-1}\big([0,1/3)\big)\subseteq\mathrm{int}_{\beta X}\mathrm{cl}_{\beta X}\big(X\cap f^{-1}\big([0,1/3)\big)\big),\]
since
\[\mathrm{cl}_{\beta X}\big(X\cap f^{-1}\big([0,1/3)\big)\big)=\mathrm{cl}_{\beta X}f^{-1}\big([0,1/3)\big),\]
because $X$ is dense in $\beta X$ and $f^{-1}([0,1/3))$ is open in $\beta X$. This contradicts the choice of $E$. It remains to show that the remainder $X\setminus(C\cup D)$ is in $\mathscr{H}$. Note that
\[f^{-1}\big([1/3,2/3]\big)\subseteq\lambda_{\mathscr H} X,\]
as $f$ is either $0$ or $1$ on $\beta X\setminus\lambda_{\mathscr H} X$ by the way it is defined. Since $f^{-1}([1/3,2/3])$ is compact, using the definition of $\lambda_{\mathscr H} X$ we have
\begin{equation}\label{JY}
f^{-1}\big([1/3,2/3]\big)\subseteq\mathrm{int}_{\beta X}\mathrm{cl}_{\beta X}A_1\cup\cdots\cup\mathrm{int}_{\beta X}\mathrm{cl}_{\beta X}A_n,
\end{equation}
where $A_i$ is a subset of $X$ whose closure $\mathrm{cl}_XA_i$ is contained in an element $H_i$ of ${\mathscr H}$ for each $i=1,\ldots,n$. We now intersect (\ref{JY}) with $X$ to obtain
\begin{eqnarray*}
X\setminus(C\cup D)&\subseteq&X\cap f^{-1}\big([1/3,2/3]\big)\\&\subseteq&\mathrm{cl}_XA_1\cup\cdots\cup\mathrm{cl}_XA_n\subseteq H_1\cup\cdots\cup H_n.
\end{eqnarray*}
Thus $X\setminus(C\cup D)$ is contained in an element of $\mathscr{H}$ and is therefore in $\mathscr{H}$.

(2) \textit{implies} (1). Suppose that $X$ is not connected modulo $\mathscr{H}$. We show that $\beta X\setminus\lambda_{\mathscr H} X$ is not connected. Let $C$ and $D$ be a separation for $X$ modulo $\mathscr{H}$. The sets $C$ and $D$ are completely separated in $X$, thus, there exists a continuous mapping $f:X\rightarrow[0,1]$ with
\[f|_C=\textbf{0}\quad\text{and}\quad f|_D=\textbf{1}.\]
Let $f_\beta:\beta X\rightarrow[0,1]$ be the continuous extension of $f$. Let $0<r<s<1$. Let
\[A=f^{-1}\big((r,s)\big).\]
Then $\mathrm{cl}_XA$ is contained in an element of $\mathscr{H}$, namely $X\setminus(C\cup D)$, as $\mathrm{cl}_XA$ is contained in $f^{-1}([r,s])$ and the latter is contained in $X\setminus(C\cup D)$ by the way we have defined $f$. Thus
\[\mathrm{int}_{\beta X}\mathrm{cl}_{\beta X}A\subseteq\lambda_{\mathscr H} X\]
by the definition of $\lambda_{\mathscr H} X$. Arguing as in the above part, we can check that
\[f_\beta^{-1}\big((r,s)\big)\subseteq\mathrm{int}_{\beta X}\mathrm{cl}_{\beta X}A.\]
Thus
\[f_\beta^{-1}\big((r,s)\big)\subseteq\lambda_{\mathscr H} X.\]
But this holds for every $0<r<s<1$, therefore
\[f_\beta^{-1}\big((0,1)\big)\subseteq\lambda_{\mathscr H} X.\]
This, in particular, implies that
\[\beta X\setminus\lambda_{\mathscr H} X=\big(f_\beta^{-1}(0)\setminus\lambda_{\mathscr H} X\big)\cup\big(f_\beta^{-1}(1)\setminus\lambda_{\mathscr H} X\big).\]
As we now check, the sets
\[K=f_\beta^{-1}(0)\setminus\lambda_{\mathscr H} X\quad\text{and}\quad L=f_\beta^{-1}(1)\setminus\lambda_{\mathscr H} X\]
are neither empty. Since $K$ and $L$ are both closed in $\beta X\setminus\lambda_{\mathscr H} X$, this will imply that $\beta X\setminus\lambda_{\mathscr H} X$ is not connected. To show this, suppose to the contrary that $f_\beta^{-1}(0)\setminus\lambda_{\mathscr H} X$ is empty. (The proof that $f_\beta^{-1}(1)\setminus\lambda_{\mathscr H} X$ is non-empty is analogous.) Then $f_\beta^{-1}(0)$ is contained in $\lambda_{\mathscr H} X$ and thus, since $f_\beta^{-1}(0)$ is compact, using the definition of $\lambda_{\mathscr H} X$ we have
\begin{equation}\label{HF}
f_\beta^{-1}(0)\subseteq\mathrm{int}_{\beta X}\mathrm{cl}_{\beta X}A_1\cup\cdots\cup\mathrm{int}_{\beta X}\mathrm{cl}_{\beta X}A_n,
\end{equation}
where $A_i$ is a subset of $X$ such that $\mathrm{cl}_XA_i$ is contained in an element $H_i$ of ${\mathscr H}$ for each $i=1,\ldots,n$. Now, we intersect (\ref{HF}) with $X$ to obtain
\begin{eqnarray*}
f^{-1}(0)=X\cap f_\beta^{-1}(0)\subseteq\mathrm{cl}_XA_1\cup\cdots\cup\mathrm{cl}_XA_n\subseteq H_1\cup\cdots\cup H_n.
\end{eqnarray*}
Therefore, since $C$ is contained in $f^{-1}(0)$ by the definition of $f$, the set $C$ is contained in an element of ${\mathscr H}$, and is thus in ${\mathscr H}$. But, this is a contradiction.
\end{proof}

\begin{remark}
Note that for a completely regular space $X$ we have $\lambda_{\mathscr H}X=\emptyset$ in the case when ${\mathscr H}$ is the trivial ideal $\{\emptyset\}$. In this case Theorem \ref{BBV} reduces to the nearly trivial assertion: \textit{$X$ is connected if and only if $\beta X$ is connected}.
\end{remark}

In Theorem \ref{BBV} it might be of some interest to know when $\beta X\setminus\lambda_{\mathscr H} X$ is contained in $\beta X\setminus X$. This will be the context of our next result.

\begin{definition}\label{JHGF}
Let $X$ be a space and let $\mathscr{H}$ be an ideal in $X$. The space $X$ is called \textit{local modulo $\mathscr{H}$} if for every $x$ in $X$ there is open neighborhood $U$ of $x$ in $X$ whose closure $\mathrm{cl}_XU$ is in $\mathscr{H}$.
\end{definition}

\begin{proposition}\label{JJHGH}
Let $X$ be a completely regular space and let $\mathscr{H}$ be an ideal in $X$. The following are equivalent.
\begin{itemize}
\item[\rm(1)] $X$ is local modulo $\mathscr{H}$.
\item[\rm(2)] $\beta X\setminus\lambda_{\mathscr H} X$ is contained in $\beta X\setminus X$.
\end{itemize}
\end{proposition}

\begin{proof}
(1) \textit{implies} (2). Let $x$ be in $X$. Let $U$ be an open neighborhood of $x$ in $X$ whose closure $\mathrm{cl}_XU$ is in $\mathscr{H}$. Let $f:X\rightarrow[0,1]$ be a continuous mapping such that
\[f(x)=0\quad\text{and}\quad f|_{X\setminus U}=\mathbf{1}.\]
Let $f_\beta:\beta X\rightarrow[0,1]$ denote the continuous extension of $f$. Let $0<r<1$. Observe that $x$ is in $f_\beta^{-1}([0,r))$,
\[f_\beta^{-1}\big([0,r)\big)\subseteq\mathrm{int}_{\beta X}\mathrm{cl}_{\beta X}f_\beta^{-1}\big([0,r)\big)\]
trivially,
\[\mathrm{cl}_{\beta X}f_\beta^{-1}\big([0,r)\big)=\mathrm{cl}_{\beta X}\big(X\cap f_\beta^{-1}\big([0,r)\big)\big),\]
as $f_\beta^{-1}([0,r))$ is open in $\beta X$ and $X$ is dense in $\beta X$, and
\[X\cap f_\beta^{-1}\big([0,r)\big)=f^{-1}\big([0,r)\big)\subseteq U\]
by the definition of $f$. Therefore $x$ is in $\mathrm{int}_{\beta X}\mathrm{cl}_{\beta X}U$. But the latter is contained in $\lambda_{\mathscr H} X$ by the definition of $\lambda_{\mathscr H} X$. Thus $x$ is in $\lambda_{\mathscr H} X$.

(2) \textit{implies} (1). Let $x$ be in $X$. Then $x$ is in $\lambda_{\mathscr H} X$, and thus $x$ is in $\mathrm{int}_{\beta X}\mathrm{cl}_{\beta X}A$ for some subset $A$ of $X$ whose closure $\mathrm{cl}_XA$ is in $\mathscr{H}$. Let
\[U=X\cap\mathrm{int}_{\beta X}\mathrm{cl}_{\beta X}A.\]
Then $U$ is an open neighborhood of $x$ in $X$. Also, $\mathrm{cl}_XU$ is in $\mathscr{H}$, as it is contained in $\mathrm{cl}_XA$.
\end{proof}

\section{Examples; connectedness modulo a topological property}\label{HGGGF}

In this section we provide specific examples for the general results we have obtained in the previous section. The idea is to correspond an ideal to a topological property in a natural way. A space will be then called ``connected modulo the topological property'' provided that it is connected modulo the ideal corresponded to that topological property. This will be done for a topological property which is closed hereditary and preserved under finite closed sums of subspaces in Subsection \ref{HJHGF}, for pseudocompactness in Subsection \ref{JJNBBV}, and for realcompactness in Subsection \ref{HGDFS}.

\subsection{The case of a closed hereditary topological property preserved under finite closed sums of subspaces}\label{HJHGF}
Results of this subsection are mostly found in \cite{Kou12}; they have been derived here, however, as corollaries of our general study of connectedness modulo an ideal.

Let us start with the following definition.

\begin{definition}\label{HGF}
A topological property $\mathfrak{P}$ is said to be
\begin{itemize}
  \item \textit{closed hereditary}, if any closed subspace of a space with $\mathfrak{P}$ also has $\mathfrak{P}$.
  \item \textit{preserved under finite closed sums of subspaces}, if any space which is a finite union of its closed subspaces each having $\mathfrak{P}$ also has $\mathfrak{P}$.
  \item \textit{preserved} (\textit{inversely preserved}, respectively) under a class $\mathscr{M}$ of mappings, if for every surjective mapping $f:X\rightarrow Y$ in $\mathscr{M}$, the space $Y$ ($X$, respectively) has $\mathfrak{P}$ if $X$ ($Y$, respectively) has $\mathfrak{P}$.
  \end{itemize}
\end{definition}

We will assume that every topological property is \textit{non-empty}, that is, for every topological property $\mathfrak{P}$ there indeed exists a space which has $\mathfrak{P}$. Note that for a closed hereditary topological property $\mathfrak{P}$ this implies that the empty space $\emptyset$ has $\mathfrak{P}$.

The following definition (together with its subsequent lemma) is to provide a connection between ideals and certain classes of topological properties.

\begin{definition}\label{JHG}
Let $\mathfrak{P}$ be a topological property. For a space $X$ define
\[{\mathscr H}_\mathfrak{P}(X)=\{H\subseteq X:\mathrm{cl}_XH\mbox{ has }\mathfrak{P}\}.\]
\end{definition}

\begin{lemma}\label{JKBG}
Let $\mathfrak{P}$ be a closed hereditary topological property which is preserved under finite closed sums of subspaces. Then, for a space $X$ the set  ${\mathscr H}_\mathfrak{P}(X)$ is an ideal in $X$.
\end{lemma}

\begin{proof}
Note that the empty space $\emptyset$ has $\mathfrak{P}$ and ${\mathscr H}_\mathfrak{P}(X)$ is therefore non-empty. Suppose that $H$ is in ${\mathscr H}_\mathfrak{P}(X)$ and that $A$ is a subset of $H$. Then $\mathrm{cl}_XH$ has $\mathfrak{P}$ and so does its closed subspace $\mathrm{cl}_XA$, as $\mathfrak{P}$ is closed hereditary. That is, $A$ is in ${\mathscr H}_\mathfrak{P}(X)$. Suppose that $G$ and $H$ are in ${\mathscr H}_\mathfrak{P}(X)$. Then $\mathrm{cl}_XG$ and $\mathrm{cl}_XH$ both have $\mathfrak{P}$, and so does their union \[\mathrm{cl}_XG\cup\mathrm{cl}_XH=\mathrm{cl}_X(G\cup H),\]
as $\mathfrak{P}$ is preserved under finite closed sums of subspaces. That is, $G\cup H$ is in ${\mathscr H}_\mathfrak{P}(X)$.
\end{proof}

In view of the above lemma the following definition now makes sense.

\begin{definition}\label{JHG}
Let $\mathfrak{P}$ be a closed hereditary topological property which is preserved under finite closed sums of subspaces. A space $X$ is said to be \textit{connected modulo $\mathfrak{P}$} if it is connected modulo the ideal ${\mathscr H}_\mathfrak{P}(X)$ in $X$.
\end{definition}

Observe that for a topological property $\mathfrak{P}$ which is closed hereditary and preserved under finite closed sums of subspaces, a space $X$ is connected modulo $\mathfrak{P}$ if it has $\mathfrak{P}$. This is because in this case, $X$ will be in ${\mathscr H}_\mathfrak{P}(X)$, which implies the non-existence of any mapping on $X$ which is $2$-valued modulo ${\mathscr H}_\mathfrak{P}(X)$.

Next, in a series of results we study how connectedness modulo a topological property behaves with respect to continuous mappings, unions and closures.

Recall that a mapping $\phi:X\rightarrow Y$, where $X$ and $Y$ are spaces, is called \textit{perfect}, if $\phi$ is a closed mapping and the fiber $\phi^{-1}(y)$ is compact for each $y$ in $Y$.

\begin{theorem}\label{JHFD}
Let $\mathfrak{P}$ be a closed hereditary topological property which is preserved under finite closed sums of subspaces and is both preserved and inversely preserved under perfect continuous surjections. Then, connectedness modulo $\mathfrak{P}$ is preserved under perfect continuous surjections.
\end{theorem}

\begin{proof}
Let $\phi:X\rightarrow Y$ be a perfect continuous mapping from a space $X$ which is connected modulo $\mathfrak{P}$ onto a space $Y$.
We prove that $Y$ is connected modulo $\mathfrak{P}$, that is, $Y$ is connected modulo the ideal ${\mathscr H}_\mathfrak{P}(Y)$. To prove this, by Theorem \ref{KJG}, it suffices to prove that $X$ is connected modulo $\phi^{-1}({\mathscr H}_\mathfrak{P}(Y))$. We prove the latter by verifying that
\begin{equation}\label{HUGFDF}
{\mathscr H}_\mathfrak{P}(X)=\phi^{-1}\big({\mathscr H}_\mathfrak{P}(Y)\big).
\end{equation}
Since $X$ is connected modulo ${\mathscr H}_\mathfrak{P}(X)$, this will prove the theorem.

Let $A$ be in ${\mathscr H}_\mathfrak{P}(X)$. Then $\mathrm{cl}_XA$ has $\mathfrak{P}$. Observe that the mapping
\[\phi|_{\mathrm{cl}_XA}:\mathrm{cl}_XA\longrightarrow\phi(\mathrm{cl}_XA),\]
which is obtained from $f$ by restricting it to a closed subspace of its domain, is perfect, and is clearly surjective. Thus $\phi(\mathrm{cl}_XA)$ has $\mathfrak{P}$, as $\mathfrak{P}$ is preserved under perfect continuous surjections. Note that $\phi(\mathrm{cl}_XA)$ is closed in $Y$, as $\phi$ is a closed mapping. In particular, $\phi(\mathrm{cl}_XA)$ contains $\mathrm{cl}_Y\phi(A)$ as a closed subspace. Therefore $\mathrm{cl}_Y\phi(A)$ has $\mathfrak{P}$, as $\mathfrak{P}$ is closed hereditary. Thus $\phi(A)$ is in ${\mathscr H}_\mathfrak{P}(Y)$ and therefore $A$ is in $\phi^{-1}({\mathscr H}_\mathfrak{P}(Y))$.

Next, let $B$ be in $\phi^{-1}({\mathscr H}_\mathfrak{P}(Y))$. Then $\phi(B)$ is in ${\mathscr H}_\mathfrak{P}(Y)$ and therefore $\mathrm{cl}_Y\phi(B)$ has $\mathfrak{P}$. Observe that the mapping
\[\phi|_{\phi^{-1}(\mathrm{cl}_Y\phi(B))}:\phi^{-1}\big(\mathrm{cl}_Y\phi(B)\big)\longrightarrow\mathrm{cl}_Y\phi(B)\]
is perfect, and is surjective, as $\phi$ is so. Thus $\phi^{-1}(\mathrm{cl}_Y\phi(B))$ has $\mathfrak{P}$, as $\mathfrak{P}$ is inversely preserved under perfect continuous surjections. But $\phi^{-1}(\mathrm{cl}_Y\phi(B))$ contains $\mathrm{cl}_XB$ as a closed subspace and $\mathfrak{P}$ is closed hereditary, thus $\mathrm{cl}_XB$ has $\mathfrak{P}$. Therefore $B$ is in ${\mathscr H}_\mathfrak{P}(X)$. This together with the above explanation proves (\ref{HUGFDF}).
\end{proof}

\begin{lemma}\label{KLJK}
Let $\mathfrak{P}$ be a closed hereditary topological property which is preserved under finite closed sums of subspaces. Let $X$ be a space and let $A$ be a closed subset of $X$. Then
\[{\mathscr H}_\mathfrak{P}(X)|_A={\mathscr H}_\mathfrak{P}(A).\]
\end{lemma}

\begin{proof}
Let $B$ be in ${\mathscr H}_\mathfrak{P}(X)|_A$. Then $B=C\cap A$ for some $C$ in ${\mathscr H}_\mathfrak{P}(X)$. In particular, $\mathrm{cl}_XC$ has $\mathfrak{P}$. But $\mathrm{cl}_AB=\mathrm{cl}_XB$, as $A$ is closed in $X$, and $\mathrm{cl}_XB$ has $\mathfrak{P}$, as $\mathrm{cl}_XB$ is closed in $\mathrm{cl}_XC$ and $\mathfrak{P}$ is closed hereditary. Thus $B$ is in ${\mathscr H}_\mathfrak{P}(A)$.

Now, let $B$ be in ${\mathscr H}_\mathfrak{P}(A)$. Then $\mathrm{cl}_AB$ has $\mathfrak{P}$. But $\mathrm{cl}_XB=\mathrm{cl}_AB$, thus $B$ is in ${\mathscr H}_\mathfrak{P}(X)$. Therefore $B=B\cap A$ is in ${\mathscr H}_\mathfrak{P}(X)|_A$.
\end{proof}

\begin{theorem}\label{HGDF}
Let $\mathfrak{P}$ be a closed hereditary topological property which is preserved under finite closed sums of subspaces. Let $X$ be a space such that
\[X=X_1\cup\cdots\cup X_n,\]
where $X_i$ for each $i=1,\dots,n$ is a closed subspace of $X$ which is connected modulo $\mathfrak{P}$. Suppose that
\[X_1\cap\cdots\cap X_n\]
is connected modulo $\mathfrak{P}$ and is non-$\mathfrak{P}$. Then $X$ is connected modulo $\mathfrak{P}$.
\end{theorem}

\begin{proof}
Since $X_i$ is closed in $X$ for each $i=1,\dots,n$, by Lemma \ref{KLJK} connectedness modulo ${\mathscr H}_\mathfrak{P}(X)|_{X_i}$ coincides with connectedness modulo ${\mathscr H}_\mathfrak{P}(X_i)$. Similarly, if we let \[A=X_1\cap\cdots\cap X_n,\]
then $A$ is closed in $X$ and thus connectedness modulo ${\mathscr H}_\mathfrak{P}(X)|_A$ coincides with connectedness modulo ${\mathscr H}_\mathfrak{P}(A)$. Also, $A$ is not in ${\mathscr H}_\mathfrak{P}(X)$, as $A$ is non-$\mathfrak{P}$. Theorem \ref{HFDF} now concludes the proof.
\end{proof}

In the following example we consider connectedness modulo $\mathfrak{P}$ when $\mathfrak{P}$ is compactness. This, in particular, provides us with an example showing that a space with a dense subspace which is connected modulo $\mathfrak{P}$ is not necessarily connected modulo $\mathfrak{P}$.

\begin{example}\label{JFS}
Let $\mathfrak{P}$ be compactness. It is clear that $\mathfrak{P}$ is closed hereditary and is preserved under finite closed sums of subspaces. Now, suppose that $X$ is a locally compact Hausdorff space. For simplicity of the notation denote ${\mathscr H}_\mathfrak{P}(X)$ by ${\mathscr C}$. In this case the ideal ${\mathscr C}$ consists of subsets $C$ of $X$ whose closure $\mathrm{cl}_XC$ is compact. By Theorem \ref{BBV} the space $X$ is connected modulo compactness if and only if $\beta X\setminus\lambda_{\mathscr C}X$ is connected, where, by definition
\[\lambda_{\mathscr C}X=\bigcup\{\mathrm{int}_{\beta X}\mathrm{cl}_{\beta X}C:C\mbox{ is a compact subspace of }X\}.\]
We check that
\begin{equation}\label{JHSD}
\lambda_{\mathscr C} X=X.
\end{equation}

Note that for a compact subspace $C$ of $X$ the set $\mathrm{int}_{\beta X}\mathrm{cl}_{\beta X}C$ is contained in $\mathrm{cl}_{\beta X}C=C$, and thus, is contained in $X$. Thus $\lambda_{\mathscr C}X$ is contained in $X$. Next, let $x$ be in $X$. Since $X$ is locally compact, there is an open neighborhood $U$ of $x$ in $X$ whose closure $\mathrm{cl}_XU$ is compact. Note that $U$ is open in $\beta X$, as $U$ is open in $X$ and $X$ is open in $\beta X$, since $X$ is locally compact. In particular, $U$ is contained in $\lambda_{\mathscr C}X$, as $U$ is contained in $\mathrm{int}_{\beta X}\mathrm{cl}_{\beta X}U$ and the latter is contained in $\lambda_{\mathscr C}X$ by its definition. Therefore $x$ is in $\lambda_{\mathscr C}X$. Thus $X$ is contained in $\lambda_{\mathscr C}X$. This shows (\ref{JHSD}).

That is, a locally compact Hausdorff space $X$ is connected modulo compactness if and only if $\beta X\setminus X$ is connected. This will be used in the following.

Let $n$ be a positive integer. It is known that the remainder $\beta\mathbb{R}^n\setminus\mathbb{R}^n$ is connected if and only if $n\geq2$. (See Exercise 6.L of \cite{GJ}.) Therefore, by the above discussion, the Euclidean space $\mathbb{R}^n$ is connected modulo compactness if and only if $n\geq2$.

We now construct an example of a space $Y$ which contains a dense subspace $X$ such that $X$ is connected modulo compactness, while the space $Y$ is not. Let
\[Y=\mathbb{R}\times[0,1],\]
considered as a subspace of $\mathbb{R}^2$, and let
\[X=\mathbb{R}\times(0,1).\]
Then $X$ is clearly dense in $Y$, and it is connected modulo compactness, as it is homeomorphic to $\mathbb{R}^2$. The space $Y$, however, is not connected modulo compactness; the mapping $f:Y\rightarrow[0,1]$ defined by
\[f(s,t)=\max\big\{0,\min\{s,1\}\big\}\]
is continuous and is $2$-valued modulo compactness. Indeed,
\[f^{-1}(0)=(-\infty,0]\times[0,1]\quad\text{and}\quad f^{-1}(1)=[1,\infty)\times[0,1]\]
and
\[Y\setminus\big(f^{-1}(0)\cup f^{-1}(1)\big)=(0,1)\times[0,1],\]
as one can easily check.
\end{example}

As it is pointed out in the following example the class of topological properties which satisfy the assumptions of the theorems in this subsection is quite wide and include almost all the so called ``covering properties''. (This class of topological properties has been also considered in \cite{Ko4}--\cite{Ko7}.)

\begin{example}\label{20UIHG}
Let $\mathfrak{P}$ be either compactness, the Lindel\"{o}f property, countable compactness, paracompactness, metacompactness, countable paracompactness, subparacompactness, $\theta$-refinability (or submetacompactness), the $\sigma$-para-Lindel\"{o}f property or $\omega$-boundedness. Then $\mathfrak{P}$ is a closed hereditary topological property which is preserved under finite closed sums of subspaces and is both preserved and inversely preserved under perfect continuous mappings. (See \cite{Bu}, \cite{Steph} and \cite{Va}.)
\end{example}

In \cite{Kou12}, for a topological property $\mathfrak{P}$ which is closed hereditary and preserved under finite closed sums of subspaces, we have called a space $X$ to be ``connected modulo $\mathfrak{P}$'' provided that there is no pair of disjoint cozero-sets $C$ and $D$ of $X$ such that the  closures in $X$ of neither has $\mathfrak{P}$, and $X\setminus(C\cup D)$ is contained in a cozero-set of $X$ whose closure in $X$ has $\mathfrak{P}$. In the following theorem we show that the notion ``connectedness modulo a topological property'' derived here from the more general notion ``connectedness modulo an ideal'' coincides with the one we have already introduced and studied in \cite{Kou12}.

\begin{theorem}\label{JJH}
Let $\mathfrak{P}$ be a closed hereditary topological property which is preserved under finite closed sums of subspaces. For a space $X$ the following are equivalent.
\begin{itemize}
\item[\rm(1)] $X$ is connected modulo $\mathfrak{P}$.
\item[\rm(2)] There is no pair $C$ and $D$ of disjoint cozero-sets of $X$ such that
\begin{itemize}
  \item $\mathrm{cl}_XC$ and $\mathrm{cl}_XD$ neither has $\mathfrak{P}$.
  \item $X\setminus(C\cup D)$ is contained in a cozero-set $E$ of $X$ such that $\mathrm{cl}_XE$ has $\mathfrak{P}$.
\end{itemize}
\end{itemize}
\end{theorem}

\begin{proof}
(2) \textit{implies} (1). Suppose that $X$ is not connected modulo $\mathfrak{P}$. There is a continuous mapping $f:X\rightarrow[0,1]$ which is $2$-valued modulo ${\mathscr H}_\mathfrak{P}(X)$. That is, if
\[E=X\setminus\big(f^{-1}(0)\cup f^{-1}(1)\big),\]
then the closure in $X$ of $E$ has $\mathfrak{P}$, while neither of $f^{-1}(0)$ and $f^{-1}(1)$ has $\mathfrak{P}$. Let
\[C=f^{-1}\big([0,1/2)\big)\quad\text{and}\quad D=f^{-1}\big((1/2,1]\big).\]
Then $C$ and $D$ are disjoint cozero-sets of $X$, and the closure in $X$ of neither has $\mathfrak{P}$, as $C$ and $D$ contain $f^{-1}(0)$ and $f^{-1}(1)$, respectively, and $\mathfrak{P}$ is closed hereditary. Also, $E$ is a cozero-set in $X$ and
\[X\setminus(C\cup D)=f^{-1}(1/2)\subseteq E.\]

(1) \textit{implies} (2). Suppose that there is a pair of disjoint cozero-sets $C$ and $D$ of $X$ such that the closures in $X$ of neither has $\mathfrak{P}$ and
\begin{equation}\label{KJF}
X\setminus(C\cup D)\subseteq E,
\end{equation}
where $E$ is a cozero-set of $X$ whose closure in $X$ has $\mathfrak{P}$. It is clear that $E$ is in ${\mathscr H}_\mathfrak{P}(X)$ while neither of $C$ and $D$ is so. Let $f,g:X\rightarrow[0,1]$ be continuous mappings such that
\begin{equation}\label{JIHF}
C=\mathrm{Coz}(f)\quad\text{and}\quad D=\mathrm{Coz}(g).
\end{equation}
Note that $X\setminus(C\cup D)$ and $X\setminus E$ are zero-sets in $X$, and they are disjoint by (\ref{KJF}). Thus, there exists a continuous mapping $h:X\rightarrow[0,1]$ such that
\[h^{-1}(0)=X\setminus E\quad\text{and}\quad h^{-1}(1)=X\setminus(C\cup D).\]
Define a mapping $k:X\rightarrow[0,1]$ by
\[k=\frac{f}{f+g+h}.\]
We show that $k$ is a continuous mapping which is $2$-valued modulo ${\mathscr H}_\mathfrak{P}(X)$.

First, we need to check that $k$ is well defined. Let $x$ be in $X$. Note that
\[f(x)+g(x)+h(x)\neq 0,\]
if either $f(x)\neq 0$ or $g(x)\neq 0$. If both $f(x)=0$ and $g(x)=0$, then $x$ is in neither of $\mathrm{Coz}(f)$ and $\mathrm{Coz}(g)$, and thus $x$ is in $E$ by (\ref{KJF}) and (\ref{JIHF}). But then $h(x)\neq 0$ by the definition of $h$, and thus $f(x)+g(x)+h(x)\neq 0$. This shows that $k$ is well defined.

It is clear that $0\leq k(x)\leq1$ for every $x$ in $X$, and that $k$ is continuous.

Note that $\mathrm{Z}(f)$ is contained in $k^{-1}(0)$ by the definition of $k$. But, since $C$ and $D$ are disjoint, by (\ref{JIHF}) we have
\[D\subseteq X\setminus C=X\setminus\mathrm{Coz}(f)=\mathrm{Z}(f).\]
Thus $D$ is contained in $k^{-1}(0)$. Since $D$ is not in ${\mathscr H}_\mathfrak{P}(X)$, the set $k^{-1}(0)$ is not in ${\mathscr H}_\mathfrak{P}(X)$ either.

Before we show that $k^{-1}(1)$ is not in ${\mathscr H}_\mathfrak{P}(X)$, we check that
\begin{equation}\label{HJHH}
C\setminus E\subseteq k^{-1}(1).
\end{equation}
Let $x$ be in $C\setminus E$. Then $x$ is not in $D$, as $x$ is in $C$, and $C$ and $D$ are disjoint. Thus $g(x)=0$ by (\ref{JIHF}). Also, $h(x)=0$ by the definition of $h$, as $x$ is not in $E$. Therefore, $k(x)=1$ by the definition of $k$.

To show that $h^{-1}(1)$ is not in ${\mathscr H}_\mathfrak{P}(X)$, suppose the contrary. Then $C\setminus E$ is in ${\mathscr H}_\mathfrak{P}(X)$ by (\ref{HJHH}). But $E$ is in ${\mathscr H}_\mathfrak{P}(X)$, therefore
\[C\cup E=(C\setminus E)\cup E\]
is in ${\mathscr H}_\mathfrak{P}(X)$. But then $C$ is in ${\mathscr H}_\mathfrak{P}(X)$, as it is contained in $C\cup E$. This is a contradiction.

It remains to show that $X\setminus(k^{-1}(0)\cup k^{-1}(1))$ is in ${\mathscr H}_\mathfrak{P}(X)$. Let $x$ be in $X$. Then by the definition of $k$, we have $0<k(x)<1$ if and only if we have both $f(x)\neq0$ and $g(x)+h(x)\neq0$. Note that $\mathrm{Coz}(h)=E$ by the way $h$ is defined. Therefore, using (\ref{JIHF}), we have
\begin{eqnarray*}
X\setminus\big(k^{-1}(0)\cup k^{-1}(1)\big)&\subseteq&\mathrm{Coz}(f)\cap\mathrm{Coz}(g+h)\\&=&\mathrm{Coz}(f)\cap\big(\mathrm{Coz}(g)\cup\mathrm{Coz}(h)\big)=C\cap(D\cup E).
\end{eqnarray*}
Note that
\[C\cap(D\cup E)=C\cap E,\]
as $C$ and $D$ are disjoint. Thus $X\setminus(k^{-1}(0)\cup k^{-1}(1))$ is contained in $E$, and is therefore in ${\mathscr H}_\mathfrak{P}(X)$, as $E$ is so.
\end{proof}

\subsection{The case of pseudocompactness}\label{JJNBBV}

In this subsection we will be dealing with pseudocompactness. Recall that a completely regular space $X$ is said to be \textit{pseudocompact} if every continuous mapping $f:X\rightarrow\mathbb{R}$ is bounded. Observe that pseudocompactness is not a closed hereditary topological property (although it is preserved under finite closed sums of subspaces). Thus, in particular, the results of this subsection do not follow from those of Subsection \ref{HJHGF}.

Theorems \ref{HJHD} and \ref{KHG} may also be found in \cite{Kou12}; they have been derived here, however, as corollaries of our general study of connectedness modulo an ideal.

\begin{definition}\label{JHG}
For a completely regular space $X$ define
\[{\mathscr U}_X=\langle U:\mbox{$U$ is an open subset of $X$ whose closure is pseudocompact}\rangle.\]
\end{definition}

\begin{definition}\label{JG}
A completely regular space $X$ is said to be \textit{connected modulo pseudocompactness} if it is connected modulo the ideal ${\mathscr U}_X$ of $X$.
\end{definition}

It is clear that if a space $X$ is pseudocompact then it is connected modulo pseudocompactness, as in this case $X$ itself is in ${\mathscr U}_X$, which implies the non-existence of any mapping on $X$ which is $2$-valued modulo ${\mathscr U}_X$.

In the next theorem we study how connectedness modulo pseudocompactness is preserved under a certain class of continuous mappings. We use the following simple observation.

\begin{lemma}\label{KJKJ}
Let $X$ be a completely regular space. Then
\[{\mathscr U}_X=\{A:\mbox{$A\subseteq U$ where $U$ is an open subset of $X$ whose closure is pseudocompact}\}.\]
\end{lemma}

\begin{proof}
It is clear that a subset of $X$ is in ${\mathscr U}_X$ if it is contained in an open subset $U$ of $X$ whose closure in $X$ is pseudocompact.

Let $A$ be in ${\mathscr U}_X$. Then $A$ is contained in a finite union $U=U_1\cup\cdots\cup U_n$, where $U_i$ is an open subset of $X$ whose closure in $X$ is pseudocompact for each $i=1,\ldots,n$. Observe that $U$ is an open subset of $X$ and
\[\mathrm{cl}_XU=\mathrm{cl}_XU_1\cup\cdots\cup\mathrm{cl}_XU_n\]
is pseudocompact, as it is a finite union of pseudocompact subspaces.
\end{proof}

A \textit{regular closed} subset in a space is a set which is identical to the closure of its interior. Thus, regular closed subsets of a space are closures of open subsets. It is known that pseudocompactness is hereditary with respect to regular closed subsets and is inversely preserved with respect to perfect open continuous surjections.

\begin{theorem}\label{HJHD}
Connectedness modulo pseudocompactness is preserved under perfect open continuous surjections.
\end{theorem}

\begin{proof}
Let $\phi:X\rightarrow Y$ be a perfect open continuous mapping from a space $X$ which is connected modulo pseudocompactness onto a space $Y$.
We show that $Y$ is connected modulo pseudocompactness, that is, $Y$ is a completely regular space which is connected modulo the ideal ${\mathscr U}_{\,Y}$. Observe that complete regularity is preserved under perfect open continuous surjections. Thus, to prove this, by Theorem \ref{KJG}, it suffices to prove that $X$ is connected modulo $\phi^{-1}({\mathscr U}_{\,Y})$. We prove the latter by verifying that
\begin{equation}\label{JKHJ}
{\mathscr U}_X=\phi^{-1}({\mathscr U}_{\,Y}).
\end{equation}
Since $X$ is connected modulo ${\mathscr U}_X$, this will then prove the theorem.

Let $A$ be in ${\mathscr U}_X$. By Lemma \ref{KJKJ} the set $A$ is contained in an open subset $U$ of $X$ whose closure $\mathrm{cl}_XU$ is pseudocompact. It is clear that $\phi(A)$ is then contained in $\phi(U)$. Note that $\phi(U)$ is an open subset of $Y$, as $\phi$ is an open mapping. To show that $\phi(A)$ is in ${\mathscr U}_{\,Y}$, by Lemma \ref{KJKJ}, it then suffices to show that $\mathrm{cl}_Y\phi(U)$ is pseudocompact. Note that $\phi(\mathrm{cl}_XU)$ is pseudocompact, as it is a continuous image of $\mathrm{cl}_XU$, and the latter is so. Observe that $\phi(\mathrm{cl}_XU)$ contains $\mathrm{cl}_Y\phi(U)$, as $\phi$ is closed, since it is perfect. But then $\mathrm{cl}_Y\phi(U)$ is pseudocompact, as it is regular closed in the pseudocompact space $\phi(\mathrm{cl}_XU)$ and pseudocompactness is hereditary with respect to regular closed subsets.

Next, let $B$ be in $\phi^{-1}({\mathscr U}_{\,Y})$. Then $\phi(B)$ is in ${\mathscr U}_{\,Y}$ and thus by Lemma \ref{KJKJ} the set $\phi(B)$ is contained in an open subset $V$ of $Y$ whose closure $\mathrm{cl}_YV$ is pseudocompact. Note that
\[B\subseteq\phi^{-1}\big(\phi(B)\big)\subseteq\phi^{-1}(V),\]
and that $\phi^{-1}(V)$ is an open subset of $X$. We show that $\mathrm{cl}_X\phi^{-1}(V)$ is pseudocompact; this will then, by Lemma \ref{KJKJ}, prove that $B$ is in ${\mathscr U}_X$. Note that pseudocompactness is inversely preserved under perfect open continuous surjections. The mapping
\[\psi=\phi|_{\phi^{-1}(\mathrm{cl}_YV)}:\phi^{-1}(\mathrm{cl}_YV)\longrightarrow\mathrm{cl}_YV\]
is continuous, perfect, open and surjective, as $\phi$ is so. Therefore $\phi^{-1}(\mathrm{cl}_YV)$ is pseudocompact, as it is the inverse image of the pseudocompact space $\mathrm{cl}_YV$ under $\psi$. But then $\mathrm{cl}_X\phi^{-1}(V)$ is pseudocompact, as it is regular closed in $\phi(\mathrm{cl}_YV)$.
\end{proof}

\begin{remark}
Theorem \ref{HJHD} may also be proved using the characterization we will give in Theorem \ref{KHG} (as we will prove its dual result, Theorem \ref{JJGHF}, by means of the characterization given in Theorem \ref{HGFDF}). The present proof of Theorem \ref{HJHD} is to demonstrate a direct application of our general result given in Theorem \ref{KJG}.
\end{remark}

The following provides a characterization for connectedness modulo pseudocompactness of a space in terms of connectedness of a familiar subspace of its Stone--\v{C}ech compactification.

We will need to use the following result due to A.W. Hager and D.G. Johnson in \cite{HJ}. (See also \cite{C} or Theorem 11.24 of \cite{We} for a proof.)

\begin{lemma}[Hager--Johnson \cite{HJ}]\label{A}
For an open subset $U$ of a completely regular space $X$, if $\mathrm{cl}_{\beta X} U$ is contained in $\upsilon X$ then $\mathrm{cl}_X U$ is pseudocompact.
\end{lemma}

The converse of Lemma \ref{A} holds as well, as we now easily check.

\begin{lemma}\label{ERS}
For an open subset $U$ of a completely regular space $X$ the closure $\mathrm{cl}_X U$ is pseudocompact if and only if $\mathrm{cl}_{\beta X} U$ is contained in $\upsilon X$.
\end{lemma}

\begin{proof}
Necessity follows from Lemma \ref{A}. For sufficiency, note that if $\mathrm{cl}_X U$ is pseudocompact then so is its closure $\mathrm{cl}_{\upsilon X} U$. But $\mathrm{cl}_{\upsilon X} U$ is also realcompact, as it is closed in $\upsilon X$ and realcompactness is closed hereditary. Note that realcompact pseudocompact spaces are compact. Thus $\mathrm{cl}_{\upsilon X} U$ is compact. It is then clear that $\mathrm{cl}_{\upsilon X} U$ contains $\mathrm{cl}_{\beta X} U$ and is contained in $\upsilon X$.
\end{proof}

\begin{theorem}\label{KHG}
For a completely regular space $X$ the following are equivalent.
\begin{itemize}
\item[\rm(1)] $X$ is connected modulo pseudocompactness.
\item[\rm(2)] $\mathrm{cl}_{\beta X}(\beta X\setminus\upsilon X)$ is connected.
\end{itemize}
\end{theorem}

\begin{proof}
For simplicity of the notation we denote ${\mathscr U}_X$ by ${\mathscr U}$.
To prove the theorem, by Theorem \ref{BBV}, it suffices to show that
\[\lambda_{\mathscr U} X=\mathrm{int}_{\beta X}\upsilon X\]
and note that
\[\mathrm{cl}_{\beta X}(\beta X\setminus\upsilon X)=\beta X\setminus\mathrm{int}_{\beta X}\upsilon X.\]

Let $z$ be in $\mathrm{int}_{\beta X}\upsilon X$. Let $U$ and $V$  be open subsets of $\beta X$ such that
\[z\in U\subseteq\mathrm{cl}_{\beta X}U\subseteq V\subseteq\mathrm{cl}_{\beta X}V\subseteq\mathrm{int}_{\beta X}\upsilon X.\]
Let
\[U'=U\cap X\quad\text{and}\quad V'=V\cap X.\]
Note that $V'$ is open in $X$ and $\mathrm{cl}_X V'$ is pseudocompact by Lemma \ref{ERS}, as $\mathrm{cl}_{\beta X}V'$ is contained in $\mathrm{cl}_{\beta X}V$ and the latter is contained in $\upsilon X$. Therefore $V'$ is in ${\mathscr U}$, and thus so is $\mathrm{cl}_X U'$, as $\mathrm{cl}_X U'$ is contained in $V'$. Therefore $\mathrm{int}_{\beta X}\mathrm{cl}_{\beta X}U'$ is contained in $\lambda_{\mathscr U} X$ by the definition of $\lambda_{\mathscr U} X$. Observe that $z$ is in $\mathrm{int}_{\beta X}\mathrm{cl}_{\beta X}U$, as $z$ is in $U$ and $U$ is open in $\beta X$. Also $\mathrm{cl}_{\beta X}U=\mathrm{cl}_{\beta X}(U\cap X)$, as $U$ is open in $\beta X$ and $X$ is dense in $\beta X$, that is $\mathrm{cl}_{\beta X}U=\mathrm{cl}_{\beta X}U'$. In particular, $z$ is in $\mathrm{int}_{\beta X}\mathrm{cl}_{\beta X}U'$ and is thus in $\lambda_{\mathscr U} X$.

Next, let $t$ be in $\lambda_{\mathscr U} X$. Then $t$ is in $\mathrm{int}_{\beta X}\mathrm{cl}_{\beta X}A$ for some subset $A$ of $X$ whose closure $\mathrm{cl}_XA$ is in ${\mathscr U}$. By Lemma \ref{KJKJ} there is an open subset $U$ of $X$ whose closure $\mathrm{cl}_X U$ is pseudocompact and contains $\mathrm{cl}_XA$. Observe that $\mathrm{cl}_{\beta X}U$, and thus in particular $\mathrm{cl}_{\beta X}A$, is contained in $\upsilon X$ by Lemma \ref{ERS}. Therefore $\mathrm{int}_{\beta X}\mathrm{cl}_{\beta X}A$ is contained in $\mathrm{int}_{\beta X}\upsilon X$. Thus $t$ is in $\mathrm{int}_{\beta X}\mathrm{cl}_{\beta X}\upsilon X$.
\end{proof}

Recall that a completely regular space $X$ is called \textit{locally pseudocompact} if for every $x$ in $X$ there is an open neighborhood $U$ of $x$ in $X$ whose closure $\mathrm{cl}_XU$ is pseudocompact. The following is known (see \cite{C}, also \cite{Ha}); it connects with Theorem \ref{KHG}.

\begin{proposition}\label{HJIIG}
For a completely regular space $X$ the following are equivalent.
\begin{itemize}
\item[\rm(1)] $X$ is locally pseudocompact.
\item[\rm(2)] $\mathrm{cl}_{\beta X}(\beta X\setminus\upsilon X)$ is contained in $\beta X\setminus X$.
\end{itemize}
\end{proposition}

\begin{proof}
Observe that by regularity of $X$ every open neighborhood of a point $x$ in $X$ whose closure in $X$ is pseudocompact contains the closure in $X$ of an open neighborhood of $x$ in $X$ which is necessarily pseudocompact, as pseudocompactness is hereditary with respect to regular closed subsets. Now, by the representation given for ${\mathscr U}_X$ in Lemma \ref{KJKJ}, it is easy to check that $X$ is locally pseudocompact if and only if $X$ is local modulo the ideal ${\mathscr U}_X$ of $X$. Proposition \ref{JJHGH} now concludes the proof.
\end{proof}

Next, we examine how connectedness modulo pseudocompactness is preserved under formation of unions. We will need the following lemma.

Recall that a subspace $A$ of a space $X$ is said to be \textit{$C$-embedded in $X$} (\textit{$C^*$-embedded in $X$}, respectively) if every continuous (continuous bounded, respectively) mapping $f:A\rightarrow\mathbb{R}$ can be continuously extended to the whole space $X$. Note that in  normal spaces closed subspaces are both $C$-embedded and $C^*$-embedded.

\begin{lemma}\label{DDJD}
Let $X$ be a completely regular space and let $A$ be a subspace of $X$.
\begin{itemize}
\item[\rm(1)] If $A$ is $C^*$-embedded in $X$ then $\mathrm{cl}_{\beta X}A=\beta A$.
\item[\rm(2)] \emph{(Gillman and Jerison \cite{GJ})} If $A$ is $C$-embedded in $X$ then $\mathrm{cl}_{\upsilon X}A=\upsilon A$.
\end{itemize}
\end{lemma}

\begin{theorem}\label{KJHGF}
Let $X$ be a normal space such that
\[X=X_1\cup\cdots\cup X_n,\]
where $X_i$ for each $i=1,\dots,n$ is a closed subspace of $X$ which is connected modulo pseudocompactness. Suppose that
\[X_1\cap\cdots\cap X_n\]
is non-pseudocompact. Then $X$ is connected modulo pseudocompactness.
\end{theorem}

\begin{proof}
Let $T$ be a closed subspace of $X$. Then
\[\beta T=\mathrm{cl}_{\beta X}T\quad\text{and}\quad\upsilon T=\mathrm{cl}_{\upsilon X}T\]
by Lemma \ref{DDJD}, as $T$ is both $C$-embedded and $C^*$-embedded in $X$, since $T$ is closed in $X$ and $X$ is normal. Therefore
\[\beta T\setminus\upsilon T=\mathrm{cl}_{\beta X}T\setminus\mathrm{cl}_{\upsilon X}T=\mathrm{cl}_{\beta X}T\setminus(\upsilon X\cap\mathrm{cl}_{\beta X}T)=\mathrm{cl}_{\beta X}T\setminus\upsilon X,\]
and thus, in particular,
\begin{eqnarray*}\label{KGF}
\mathrm{cl}_{\beta T}(\beta T\setminus\upsilon T)&=&\mathrm{cl}_{\mathrm{cl}_{\beta X}T}(\mathrm{cl}_{\beta X}T\setminus\upsilon X)\\&=&\mathrm{cl}_{\beta X}T\cap\mathrm{cl}_{\beta X}(\mathrm{cl}_{\beta X}T\setminus\upsilon X)=\mathrm{cl}_{\beta X}(\mathrm{cl}_{\beta X}T\setminus\upsilon X).
\end{eqnarray*}

Let
\[A=X_1\cap\cdots\cap X_n.\]
Then $A$ is closed in $X$, as it is the intersection of a finite number of closed subsets, and $\beta A\setminus\upsilon A$, which equals to $\mathrm{cl}_{\beta X}A\setminus\upsilon X$ by the above observation, is non-empty, as $A$ is non-pseudocompact. Note that $\mathrm{cl}_{\beta X_i}(\beta X_i\setminus\upsilon X_i)$, which equals to $\mathrm{cl}_{\beta X}(\mathrm{cl}_{\beta X}X_i\setminus\upsilon X)$ by the above observation, is connected for each $i=1,\dots,n$ by Theorem \ref{KHG}, as $X_i$ is connected modulo pseudocompactness. Observe that $\mathrm{cl}_{\beta X}(\mathrm{cl}_{\beta X}X_i\setminus\upsilon X)$ contains the non-empty set $\mathrm{cl}_{\beta X}A\setminus\upsilon X$ for each $i=1,\dots,n$. Therefore, the union
\[\bigcup_{i=1}^n\mathrm{cl}_{\beta X}(\mathrm{cl}_{\beta X}X_i\setminus\upsilon X)\]
is connected. Note that
\[\bigcup_{i=1}^n\mathrm{cl}_{\beta X}(\mathrm{cl}_{\beta X}X_i\setminus\upsilon X)=\mathrm{cl}_{\beta X}\bigg(\bigcup_{i=1}^n\mathrm{cl}_{\beta X}X_i\setminus\upsilon X\bigg)=\mathrm{cl}_{\beta X}(\beta X\setminus\upsilon X).\]
Therefore, $X$ is connected modulo pseudocompactness by Theorem \ref{KHG}.
\end{proof}

In the following example we show that the conclusion in the above theorem fails if the number of $X_i$'s is infinite.

\begin{example}\label{JGF}
Let $X$ be the space $\mathbb{R}$ endowed with the standard topology and let $X_n=(-\infty,n]$ for each positive integer $n$. Then $X$ is normal and
\[X=\bigcup_{n=1}^\infty X_n.\]
Also, $X_n$, for each positive integer $n$, is closed in $X$ and is connected modulo pseudocompactness. (The latter follows from Theorem \ref{KHG}. To see this, observe that $\upsilon X_n=X_n$, as $X_n$ is realcompact. The space $\mathrm{cl}_{\beta X_n}(\beta X_n\setminus\upsilon X_n)$ is then homeomorphic to
\[\mathrm{cl}_{\beta(-\infty,0]}\big(\beta(-\infty,0]\setminus(-\infty,0]\big)=\beta(-\infty,0]\setminus(-\infty,0]\]
which is known to be connected; see Exercise 6.L of \cite{GJ}.) Note that
\[\bigcap_{n=1}^\infty X_n=(-\infty,1]\]
is clearly non-pseudocompact. Finally, $\mathbb{R}$ is not connected modulo pseudocompactness by Theorem \ref{KHG}, as $\beta\mathbb{R}\setminus\mathbb{R}$ (which equals to $\mathrm{cl}_{\beta\mathbb{R}}(\beta\mathbb{R}\setminus\mathbb{R})$) is not connected. (See Exercise 6.L of \cite{GJ}.)
\end{example}

The following example shows that a space which has a dense subspace which is connected modulo pseudocompactness may not be itself connected modulo pseudocompactness.

\begin{example}\label{HGG}
Let
\[Y=\mathbb{R}\times[0,1],\]
considered as a subspace of $\mathbb{R}^2$, and let
\[X=\mathbb{R}\times(0,1).\]
Then $X$ is dense in $Y$. Also, $X$ is connected modulo pseudocompactness, as $X$ is homeomorphic to $\mathbb{R}^2$ and the latter is so by Theorem \ref{KHG}, since $\beta\mathbb{R}^2\setminus\mathbb{R}^2$ (which equals to $\mathrm{cl}_{\beta\mathbb{R}^2}(\beta\mathbb{R}^2\setminus\mathbb{R}^2)$) is connected. The space $Y$, however, is not connected modulo pseudocompactness; the mapping $f:Y\rightarrow[0,1]$ defined by
\[f(s,t)=\max\big\{0,\min\{s,1\}\big\}\]
is continuous and is $2$-valued modulo pseudocompactness, as we now check. Suppose that either
\[f^{-1}(0)=(-\infty,0]\times[0,1]\quad\text{or}\quad f^{-1}(1)=[1,\infty)\times[0,1],\]
say the latter, is in ${\mathscr U}_{\,Y}$. By Lemma \ref{KJKJ} then $f^{-1}(1)$ is contained in an open subset of $Y$ whose closure in $Y$ is pseudocompact. But then $f^{-1}(1)$ is pseudocompact, as in this case $f^{-1}(1)$ is a regular closed subset of $Y$, and pseudocompactness is hereditary with respect to regular closed subsets. This is a contradiction. Observe that
\[Y\setminus\big(f^{-1}(0)\cup f^{-1}(1)\big)=(0,1)\times[0,1],\]
is in ${\mathscr U}_{\,Y}$, as it is an open subset of $Y$ whose closure in $Y$ is pseudocompact (indeed compact).
\end{example}

In the following example we show that the product of two spaces which are connected modulo pseudocompactness is not necessarily connected modulo pseudocompactness.

\begin{example}\label{HGHFDF}
Let
\[X=\{0,1\}\times[0,\infty),\]
where $\{0,1\}$ and $[0,\infty)$ are considered as subspaces of $\mathbb{R}$ endowed with the standard topology. The spaces $\{0,1\}$ and $[0,\infty)$ are connected modulo pseudocompactness. (That $[0,\infty)$ is connected modulo pseudocompactness follows from Theorem \ref{KHG} and the fact that $\beta[0,\infty)\setminus[0,\infty)$ is connected.) The space $X$, however, is not connected modulo pseudocompactness; let $f:X\rightarrow[0,1]$ be defined such that
\[f|_{\{0\}\times[0,\infty)}=\mathbf{0}\quad\text{and}\quad f|_{\{1\}\times[0,\infty)}=\mathbf{1}.\]
Then $f$ is continuous, and is $2$-valued modulo ${\mathscr U}_X$, as neither of
\[f^{-1}(0)={\{0\}\times[0,\infty)}\quad\text{and}\quad f^{-1}(1)={\{1\}\times[0,\infty)}\]
is in ${\mathscr U}_X$. (Note that $f^{-1}(0)$ and $f^{-1}(1)$ are both open and closed in $X$. Thus, if either $f^{-1}(0)$ or $f^{-1}(1)$ is in ${\mathscr U}_X$, then, using Lemma \ref{KJKJ}, it will be contained in an open subset of $X$ whose closure in $X$ is pseudocompact, and is therefore pseudocompact. But, this is clearly not correct.)
\end{example}

\subsection{The case of realcompactness}\label{HGDFS}

In this subsection we will be dealing with realcompactness.  Recall that a space is called \textit{realcompact} if it is homeomorphic to a closed subspace of some product $\mathbb{R}^\alpha$.

\begin{definition}\label{KJH}
For a normal space $X$ define
\[{\mathscr R}_X=\langle R:\mbox{$R$ is a subset of $X$ whose closure is realcompact}\rangle.\]
\end{definition}

\begin{definition}\label{JHD}
A normal space $X$ is said to be \textit{connected modulo realcompactness} if it is connected modulo the ideal ${\mathscr R}_X$ of $X$.
\end{definition}

The next theorem provides a characterization for connectedness modulo realcompactness of a space in terms of connectedness of a familiar subspace of its Stone--\v{C}ech compactification. We will need a few lemmas first.

The next lemma states that realcompactness is preserved under finite closed sums of subspaces in the realm of normal spaces. (Realcompactness, despite the fact that it is a closed hereditary topological property, is not in general preserved under finite closed sums of subspaces in the realm of completely regular spaces. In \cite{M} -- a correction in \cite{M1} -- S. Mr\'{o}wka describes a completely regular space which is not realcompact but it can be represented as the union of two closed realcompact subspaces; a simpler such example is given by A. Mysior in \cite{My}.)

\begin{lemma}\label{JJHG}
Let $X$ be a normal space such that
\[X=X_1\cup\cdots\cup X_n,\]
where $X_i$ is a closed realcompact subspace of $X$ for each $i=1,\ldots,n$. Then $X$ is realcompact.
\end{lemma}

\begin{proof}
Let $i=1,\ldots,n$. Observe that $\upsilon X_i=X_i$, as $X_i$ is realcompact. Also, $\mathrm{cl}_{\upsilon X}X_i=\upsilon X_i$ by Lemma \ref{DDJD}, as $X_i$ is $C$-embedded in $X$, since it is closed in $X$ and $X$ is normal. Thus
\[\upsilon X=\mathrm{cl}_{\upsilon X}X_1\cup\cdots\cup\mathrm{cl}_{\upsilon X} X_n=\upsilon X_1\cup\cdots\cup\upsilon X_n=X_1\cup\cdots\cup X_n=X.\]
That is, $X$ is realcompact.
\end{proof}

\begin{lemma}\label{JHFD}
Let $X$ be a normal space. Then
\[{\mathscr R}_X=\{A:\mbox{$A$ is a subset of $X$ whose closure is realcompact}\}.\]
\end{lemma}

\begin{proof}
It is clear that a subset of $X$ is in ${\mathscr R}_X$ if its closure in $X$ is realcompact.

Let $A$ be in ${\mathscr R}_X$. Then $A$ is contained in a finite union
\[R=R_1\cup\cdots\cup R_n,\]
where $R_i$ is a subset of $X$ whose closure in $X$ is realcompact for each $i=1,\ldots,n$. Observe that $\mathrm{cl}_XR$ is normal, as it is closed in $X$ and $X$ is so. Therefore
\[\mathrm{cl}_XR=\mathrm{cl}_XR_1\cup\cdots\cup\mathrm{cl}_XR_n\]
is realcompact by Lemma \ref{JJHG}, as it is a union of a finite number of its closed realcompact subspaces. The closure $\mathrm{cl}_XA$ is then realcompact, as it is a closed subspace of $\mathrm{cl}_XR$.
\end{proof}

\begin{lemma}\label{B}
Let $X$ be a completely regular space. Let $A$ be a subset of $X$ such that
\[\mathrm{cl}_{\beta X}A\cap(\upsilon X\setminus X)=\emptyset.\]
Then $\mathrm{cl}_XA$ is realcompact.
\end{lemma}

\begin{proof}
We have
\begin{eqnarray*}
\mathrm{cl}_XA=\mathrm{cl}_{\beta X} A\cap X&=&(\mathrm{cl}_{\beta X} A\cap X)\cup\big(\mathrm{cl}_{\beta X} A\cap(\upsilon X\setminus X)\big)\\&=&\mathrm{cl}_{\beta X} A\cap\big(X\cup(\upsilon X\setminus X)\big)=\mathrm{cl}_{\beta X} A\cap\upsilon X=\mathrm{cl}_{\upsilon X} A.
\end{eqnarray*}
Thus $\mathrm{cl}_XA$ is closed in $\upsilon X$, and therefore is realcompact.
\end{proof}

The following lemma provides sort of converse for Lemma \ref{B}.

\begin{lemma}\label{JGH}
Let $X$ be a normal space and let $A$ be a closed realcompact subspace of $X$. Then
\[\mathrm{int}_{\beta X}\mathrm{cl}_{\beta X}A\cap\mathrm{cl}_{\beta X}(\upsilon X\setminus X)=\emptyset.\]
\end{lemma}

\begin{proof}
Suppose otherwise. Thus, in particular
\begin{equation}\label{LK}
\mathrm{int}_{\beta X}\mathrm{cl}_{\beta X}A\cap(\upsilon X\setminus X)\neq\emptyset.
\end{equation}
Note that $\mathrm{cl}_{\upsilon X}A=\upsilon A$ by Lemma \ref{DDJD}, as $A$ is $C$-embedded in $X$, since $A$ is closed in $X$
and $X$ is normal. Also, $\upsilon A=A$, as $A$ is realcompact. Therefore
\[\mathrm{cl}_{\beta X}A\cap(\upsilon X\setminus X)=\mathrm{cl}_{\upsilon X}A\setminus X=\upsilon A\setminus X=A\setminus X=\emptyset.\]
This contradicts (\ref{LK}).
\end{proof}

\begin{theorem}\label{HGFDF}
For a normal space $X$ the following are equivalent.
\begin{itemize}
\item[\rm(1)] $X$ is connected modulo realcompactness.
\item[\rm(2)] $\mathrm{cl}_{\beta X}(\upsilon X\setminus X)$ is connected.
\end{itemize}
\end{theorem}

\begin{proof}
For simplicity of the notation we denote ${\mathscr R}_X$ by ${\mathscr R}$.
To prove the theorem, by Theorem \ref{BBV}, it suffices to show that
\[\lambda_{\mathscr R} X=\beta X\setminus\mathrm{cl}_{\beta X}(\upsilon X\setminus X).\]

Let $z$ be in $\beta X\setminus\mathrm{cl}_{\beta X}(\upsilon X\setminus X)$. Let $U$ be open neighborhood of $z$ in $\beta X$ such that
\[\mathrm{cl}_{\beta X}U\cap\mathrm{cl}_{\beta X}(\upsilon X\setminus X)=\emptyset.\]
Let $V=U\cap X$. Note that $\mathrm{cl}_{\beta X}U=\mathrm{cl}_{\beta X}V$, as $U$ is open $\beta X$ and $X$ is dense in $\beta X$. Thus $\mathrm{cl}_XV$ is in ${\mathscr R}$, as it is realcompact by Lemma \ref{B}. Therefore $\mathrm{int}_{\beta X}\mathrm{cl}_{\beta X}V$ is contained in $\lambda_{\mathscr R} X$ by the definition of $\lambda_{\mathscr R} X$. Observe that $z$ is in $\mathrm{int}_{\beta X}\mathrm{cl}_{\beta X}U$, as $z$ is in $U$ and $U$ is open in $\beta X$. In particular, $z$ is in $\mathrm{int}_{\beta X}\mathrm{cl}_{\beta X}V$, and is thus in $\lambda_{\mathscr R} X$.

Next, let $t$ be in $\lambda_{\mathscr R} X$. Then $t$ is in $\mathrm{int}_{\beta X}\mathrm{cl}_{\beta X}A$ for some subset $A$ of $X$ whose closure $\mathrm{cl}_XA$ is in ${\mathscr R}$. The set $\mathrm{cl}_XA$ is then a realcompact subset of $X$ by Lemma \ref{JHFD}. Therefore
\[\mathrm{int}_{\beta X}\mathrm{cl}_{\beta X}A\cap\mathrm{cl}_{\beta X}(\upsilon X\setminus X)=\emptyset\]
by Lemma \ref{JGH}. That is, $\mathrm{int}_{\beta X}\mathrm{cl}_{\beta X}A$ is contained in $\beta X\setminus\mathrm{cl}_{\beta X}(\upsilon X\setminus X)$. Thus $t$ is in $\beta X\setminus\mathrm{cl}_{\beta X}(\upsilon X\setminus X)$.
\end{proof}

Recall that a completely regular space $X$ is called \textit{locally realcompact} if for every $x$ in $X$ there is an open neighborhood $U$ of $x$ in $X$ whose closure $\mathrm{cl}_XU$ is realcompact. For a normal space $X$ (using the representation given for ${\mathscr R}_X$ in Lemma \ref{JHFD}) it is easy to check that $X$ is locally realcompact if and only if it is local modulo the ideal ${\mathscr R}_X$ of $X$. The following now follows from Proposition \ref{JJHGH}.

\begin{proposition}\label{KJGG}
For a normal space $X$ the following are equivalent.
\begin{itemize}
\item[\rm(1)] $X$ is locally realcompact.
\item[\rm(2)] $\mathrm{cl}_{\beta X}(\upsilon X\setminus X)$ is contained in $\beta X\setminus X$.
\end{itemize}
\end{proposition}

\begin{remark}\label{NJHG}
In \cite{MRW}, J. Mack, M. Rayburn and R.G. Woods have observed that a completely regular space $X$ is locally realcompact if and only if $\upsilon X$ is open in $\beta X$. From this, Proposition \ref{KJGG} follows easily. (See \cite{Ko7}.)
\end{remark}

We next examine how connectedness modulo realcompactness is preserved under a certain class of continuous mappings. We need a few preliminaries first.

Let $X$ and $Y$ be completely regular spaces. A continuous mapping $\phi:X\rightarrow Y$ is said to be \textit{hyper-real} if
\[\phi_\beta(\beta X\setminus\upsilon X)\subseteq\beta Y\setminus\upsilon Y,\]
where $\phi_\beta:\beta X\rightarrow\beta Y$ is the continuous extension of $\phi$. Hyper-real mappings were originally considered by R.L. Blair in his unpublished manuscript \cite{Bl}, where he used them as a tool to
study preservation of realcompactness and inverse preservation of pseudocompactness. It is known that every perfect open continuous surjection between completely regular spaces is hyper-real. (See Corollaries 15.14 and 17.19 of \cite{We}.)

Let $X$ and $Y$ be spaces and let $\phi:X\rightarrow Y$ be a continuous mapping. Let $D$ be a dense subspace of $X$ such that $\phi|_D:D\rightarrow\phi(D)$ is perfect. Then
\[\phi(X\setminus D)\subseteq Y\setminus\phi(D).\]
(See Theorem 1.8 (i) of \cite{PW}.) In particular, if $X$ and $Y$ are completely regular spaces and $\phi:X\rightarrow Y$ is a perfect open continuous surjection, then
\[\phi_\beta(\beta X\setminus X)\subseteq\beta Y\setminus Y,\]
where $\phi_\beta:\beta X\rightarrow\beta Y$ is the continuous extension of $\phi$.

Recall that the Hewitt realcompactification $\upsilon X$ of a completely regular space $X$ may be expressed as the intersection of all cozero-sets of $\beta X$ which contain  $X$, equivalently, $\beta X\setminus\upsilon X$ is the union of all zero-sets of $\beta X$ which miss $X$.

The proof of the next theorem is motivated by the one we have given for Theorem 8.5 of \cite{Kou12}.

\begin{theorem}\label{JJGHF}
Connectedness modulo realcompactness is preserved under perfect open continuous surjections.
\end{theorem}

\begin{proof}
Let $\phi:X\rightarrow Y$ be a perfect open continuous mapping from a space $X$, which is connected modulo realcompactness, onto a space $Y$. We show that $Y$ is connected modulo realcompactness. Observe that normality is preserved under perfect continuous surjections. Thus, to prove this, by Theorem \ref{HGFDF}, it suffices to prove that $\mathrm{cl}_{\beta Y}(\upsilon Y\setminus Y)$ is connected. We prove the latter by showing that
\[\phi_\beta\big(\mathrm{cl}_{\beta X}(\upsilon X\setminus X)\big)=\mathrm{cl}_{\beta Y}(\upsilon Y\setminus Y),\]
where $\phi_\beta:\beta X\rightarrow\beta Y$ is the continuous extension of $\phi$. Since $\mathrm{cl}_{\beta X}(\upsilon X\setminus X)$ is connected by Theorem \ref{HGFDF}, this will prove the theorem.

First, we show that
\begin{equation}\label{H}
\phi_\beta(X)=Y\quad\text{and}\quad\phi_\beta^{-1}(Y)=X.
\end{equation}
That $\phi_\beta(X)=Y$ is trivial, as $\phi_\beta$ extends $\phi$ and $\phi$ is surjective. That $X$ is contained in $\phi_\beta^{-1}(Y)$ is also trivial, as $\phi_\beta$ extends $\phi$. That $\phi_\beta^{-1}(Y)$ is contained in $X$ follows from the fact that
\[\phi_\beta(\beta X\setminus X)\subseteq\beta Y\setminus Y,\]
as $X$ is dense in $\beta X$ and
\[\phi=\phi_\beta|_X:X\rightarrow \phi_\beta(X)=Y\]
is perfect. This shows (\ref{H}).

Next, we show that
\begin{equation}\label{GKH}
\phi_\beta(\beta X\setminus\upsilon X)=\beta Y\setminus\upsilon Y\quad\text{and}\quad\phi_\beta^{-1}(\beta Y\setminus\upsilon Y)=\beta X\setminus\upsilon X.
\end{equation}
Note that if $Z$ is a zero-set in $\beta Y$ which is disjoint from $Y$, then $\phi_\beta^{-1}(Z)$ is a zero-set in $\beta X$ which is disjoint from $X$, as
\[\phi_\beta^{-1}(Z)\cap X=\phi_\beta^{-1}(Z)\cap \phi_\beta^{-1}(Y)\]
by (\ref{H}). Therefore
\begin{eqnarray*}\label{GLH}
\phi_\beta^{-1}(\beta Y\setminus\upsilon Y)&=&\phi_\beta^{-1}\Big(\bigcup\big\{Z\in\mathrm{Z}(\beta Y):Z\cap Y=\emptyset\big\}\Big)\\&=&\bigcup\big\{\phi_\beta^{-1}(Z):Z\in \mathrm{Z}(\beta Y)\mbox{ and }Z\cap Y=\emptyset\big\}\\&\subseteq&\bigcup\big\{S\in\mathrm{Z}(\beta X):S\cap X=\emptyset\big\}=\beta X\setminus\upsilon X.
\end{eqnarray*}
Thus
\[\phi_\beta\big(\phi_\beta^{-1}(\beta Y\setminus\upsilon Y)\big)\subseteq \phi_\beta(\beta X\setminus\upsilon X).\]
But
\[\beta Y\setminus\upsilon Y=\phi_\beta\big(\phi_\beta^{-1}(\beta Y\setminus\upsilon Y)\big),\]
as $\phi_\beta$ is surjective, since $\phi_\beta(\beta X)$ is a compact subspace of $\beta Y$ which contains $\phi(X)=Y$. Therefore
\[\beta Y\setminus\upsilon Y\subseteq\phi_\beta(\beta X\setminus\upsilon X).\]
On the other hand
\[\phi_\beta(\beta X\setminus\upsilon X)\subseteq\beta Y\setminus\upsilon Y,\]
as $\phi$ is hyper-real, since $\phi$ is perfect open. Therefore
\[\phi_\beta(\beta X\setminus\upsilon X)=\beta Y\setminus\upsilon Y.\]
Also,
\[\beta X\setminus\upsilon X\subseteq \phi_\beta^{-1}\big(\phi_\beta(\beta X\setminus\upsilon X)\big)=\phi_\beta^{-1}(\beta Y\setminus\upsilon Y),\]
which together with the above relations shows that
\[\phi_\beta^{-1}(\beta Y\setminus\upsilon Y)=\beta X\setminus\upsilon X.\]
This proves (\ref{GKH}).

Since $\phi_\beta$ is surjective, it follows from (\ref{H}) and (\ref{GKH}) that
\begin{equation}\label{FF}
\phi_\beta(\upsilon X\setminus X)=\upsilon Y\setminus Y\quad\text{and}\quad\phi_\beta^{-1}(\upsilon Y\setminus Y)=\upsilon X\setminus X.
\end{equation}
Using (\ref{FF}) we have
\[\phi_\beta\big(\mathrm{cl}_{\beta X}(\upsilon X\setminus X)\big)\subseteq\mathrm{cl}_{\beta Y}\phi_\beta(\upsilon X\setminus X)=\mathrm{cl}_{\beta Y}(\upsilon Y\setminus Y),\]
and
\[\mathrm{cl}_{\beta Y}(\upsilon Y\setminus Y)\subseteq\phi_\beta\big(\mathrm{cl}_{\beta X}(\upsilon X\setminus X)\big),\]
as
\[\upsilon Y\setminus Y=\phi_\beta(\upsilon X\setminus X)\subseteq\phi_\beta\big(\mathrm{cl}_{\beta X}(\upsilon X\setminus X)\big),\]
and the latter is compact. Thus
\[\mathrm{cl}_{\beta Y}(\upsilon Y\setminus Y)=\phi_\beta\big(\mathrm{cl}_{\beta X}(\upsilon X\setminus X)\big),\]
as desired.
\end{proof}

Next, we examine how connectedness modulo realcompactness behaves with respect to formation of unions. The following theorem (as well as the proof) is dual to Theorem \ref{KJHGF}.

\begin{theorem}\label{JIHGF}
Let $X$ be a normal space such that
\[X=X_1\cup\cdots\cup X_n,\]
where $X_i$ for each $i=1,\dots,n$ is a closed subspace of $X$ which is connected modulo realcompactness. Suppose that
\[X_1\cap\cdots\cap X_n\]
is non-realcompact. Then $X$ is connected modulo realcompactness.
\end{theorem}

\begin{proof}
Let $T$ be a closed subspace of $X$. Observe that
\[\beta T=\mathrm{cl}_{\beta X}T\quad\text{and}\quad\upsilon T=\mathrm{cl}_{\upsilon X}T\]
by Lemma \ref{DDJD}. Therefore
\[\upsilon T\setminus T=\mathrm{cl}_{\upsilon X}T\setminus T=(\upsilon X\cap\mathrm{cl}_{\beta X}T)\setminus T=\mathrm{cl}_{\beta X}T\cap(\upsilon X\setminus X),\]
and thus
\begin{eqnarray*}\label{KGF}
\mathrm{cl}_{\beta T}(\upsilon T\setminus T)&=&\mathrm{cl}_{\mathrm{cl}_{\beta X}T}\big(\mathrm{cl}_{\beta X}T\cap(\upsilon X\setminus X)\big)\\&=&\mathrm{cl}_{\beta X}T\cap\mathrm{cl}_{\beta X}\big(\mathrm{cl}_{\beta X}T\cap(\upsilon X\setminus X)\big)=\mathrm{cl}_{\beta X}\big(\mathrm{cl}_{\beta X}T\cap(\upsilon X\setminus X)\big).
\end{eqnarray*}

Let
\[A=X_1\cap\cdots\cap X_n.\]
Observe that $A$ is closed in $X$ and $\upsilon A\setminus A$, which equals to $\mathrm{cl}_{\beta X}A\cap(\upsilon X\setminus X)$ by the above observation, is non-empty, as $A$ is non-realcompact. Note that $\mathrm{cl}_{\beta X_i}(\upsilon X_i\setminus X_i)$, which equals to
\begin{equation}\label{HGF}
\mathrm{cl}_{\beta X}\big(\mathrm{cl}_{\beta X}X_i\cap(\upsilon X\setminus X)\big)
\end{equation}
by the above observation, is connected for each $i=1,\dots,n$ by Theorem \ref{HGFDF}, as $X_i$ is connected modulo realcompactness. Observe that (\ref{HGF}) contains the non-empty set
\[\mathrm{cl}_{\beta X}A\cap(\upsilon X\setminus X)\]
for each $i=1,\dots,n$. Therefore, the union
\[\bigcup_{i=1}^n\mathrm{cl}_{\beta X}\big(\mathrm{cl}_{\beta X}X_i\cap(\upsilon X\setminus X)\big)\]
is connected. Since
\[\bigcup_{i=1}^n\mathrm{cl}_{\beta X}\big(\mathrm{cl}_{\beta X}X_i\cap(\upsilon X\setminus X)\big)=\mathrm{cl}_{\beta X}\bigg(\bigcup_{i=1}^n\mathrm{cl}_{\beta X}X_i\cap(\upsilon X\setminus X)\bigg)=\mathrm{cl}_{\beta X}(\upsilon X\setminus X),\]
Theorem \ref{HGFDF} implies that $X$ is connected modulo realcompactness.
\end{proof}

The following example shows that the conclusion in the above theorem (and also Theorem \ref{HFDF}) fails if the number of $X_i$'s is infinite.

\begin{example}\label{HGGG}
Let
\[X=A\oplus I,\]
where $A=I=[0,\Omega)$. (Here $\Omega$ is the first uncountable ordinal and $\oplus$ denotes the free union.) Let
\[X_i=A\cup\{i\}\]
for each $i$ in $I$. Then $X$ is a normal space, $X_i$ for each $i$ in $I$ is closed in $X$, $A$ is non-realcompact,
\[X=\bigcup_{i\in I}X_i\quad\text{and}\quad A=\bigcap_{i\in I}X_i.\]
We verify that $X_i$ for each $i$ in $I$ is connected modulo realcompactness, while $X$ itself is not.

Note that $\beta A=[0,\Omega]$. Thus $\upsilon A=[0,\Omega]$, as $\upsilon A\neq A$, since $A$ is not realcompact. Therefore $\upsilon A\setminus A=\{\Omega\}$. Let $i$ be in $I$. Note that $\upsilon X_i=\mathrm{cl}_{\upsilon X}X_i$ by Lemma \ref{DDJD}, as $X_i$ is $C$-embedded in $X$, since $X_i$ is closed in $X$ and $X$ is normal. Similarly, $\upsilon A=\mathrm{cl}_{\upsilon X}A$. Thus
\[\upsilon X_i\setminus X_i=\mathrm{cl}_{\upsilon X}X_i\setminus X_i=\mathrm{cl}_{\upsilon X}\big(A\cup\{i\}\big)\setminus\big(A\cup\{i\}\big)=\mathrm{cl}_{\upsilon X}A\setminus A=\upsilon A\setminus A=\{\Omega\}.\]
Therefore
\[\mathrm{cl}_{\beta X_i}(\upsilon X_i\setminus X_i)=\{\Omega\}\]
is connected. Note that $X_i$ is normal, as it is closed in $X$ and $X$ is so. Thus $X_i$ is connected modulo realcompactness by Theorem \ref{HGFDF}. To conclude the proof, observe that
\[\upsilon X\setminus X=(\mathrm{cl}_{\upsilon X}A\cup\mathrm{cl}_{\upsilon X} I)\setminus X=(\mathrm{cl}_{\upsilon X}A\setminus X)\cup(\mathrm{cl}_{\upsilon X} I\setminus X)=(\upsilon A\setminus A)\cup(\upsilon I\setminus I)\]
is a two point set, and thus so is its closure $\mathrm{cl}_{\beta X}(\upsilon X\setminus X)$. Therefore $\mathrm{cl}_{\beta X}(\upsilon X\setminus X)$ is not connected. Again, Theorem \ref{HGFDF} implies that $X$ is not connected modulo realcompactness.

The above example may also be used to show that the conclusion in Theorem \ref{HFDF} fails if the number of $X_i$'s is infinite. For this purpose, note that by an argument similar to the one in the proof of Lemma \ref{KLJK} (using Lemma \ref{JHFD} and observing that $A$ and $X_i$ are closed in $X$) we have
\[{\mathscr R}_X|_A={\mathscr R}_A\quad\text{and}\quad{\mathscr R}_X|_{X_i}={\mathscr R}_{X_i}\]
for each $i$ in $I$. Thus, in particular, $X_i$ for each $i$ in $I$ is connected modulo ${\mathscr R}_X|_{X_i}$, as it is connected modulo ${\mathscr R}_{X_i}$, and $A$ is connected modulo ${\mathscr R}_X|_A$, as it is connected modulo ${\mathscr R}_A$. Also, $A$ is not in ${\mathscr R}_X$ by Lemma \ref{JHFD}, as it is closed in $X$ and is non-realcompact.
\end{example}

The following example shows that a space with a dense subspace which is connected modulo realcompactness may fail to be connected modulo realcompactness.

\begin{example}\label{FD}
Let $X$ be the space defined in Example \ref{HGGG}. Let $D$ be the subspace of $X$ consisting of all isolated points of $X$. Then $D$ is dense in $X$ trivially, and is connected modulo realcompactness, while $X$ itself is not connected modulo realcompactness, as it is shown in Example \ref{HGGG}. (That $D$ is connected modulo realcompactness follows from the facts that $D$ is realcompact, as it is a discrete space of non-measurable cardinality, and by Theorem \ref{HGFDF} every realcompact normal space is connected modulo realcompactness.)
\end{example}

The following example shows that the product of two spaces each connected modulo realcompactness may not be connected modulo realcompactness.

\begin{example}\label{FFD}
Let $X$ be the space defined in Example \ref{HGGG}. Observe that $X$ is homeomorphic to the product space
\[[0,\Omega)\times\{0,1\}.\]
While $X$ itself is not connected modulo realcompactness by Example \ref{HGGG}, both spaces $[0,\Omega)$ and $\{0,1\}$ are connected modulo realcompactness. (That $[0,\Omega)$ is connected modulo realcompactness has been observed in Example \ref{HGGG}.)
\end{example}

\section*{Acknowledgements}

The author would like to thank the referee for his/her careful reading of the manuscript and his/her comments.

\end{document}